\documentclass[a4paper,11pt]{amsart}
\usepackage{amstext, amsfonts,amsthm, amsmath, amssymb}
\usepackage[utf8]{inputenc}
\usepackage[colorlinks,linkcolor=blue,hyperindex]{hyperref}
\usepackage[justification=centering]{caption}
\usepackage{gensymb}
\usepackage{xcolor}
\usepackage{xfrac}  
\usepackage{multicol}
\usepackage{graphicx}
\usepackage{subfigure}
\usepackage{indentfirst}
\usepackage{enumerate}
\usepackage[english]{babel}
\usepackage{fancyhdr}
\usepackage{refcount}
\usepackage[a4paper,margin=2.35cm]{geometry}
\usepackage{mathtools}
\usepackage{enumitem}
\usepackage{tabularx}
\usepackage{booktabs}
\usepackage{placeins}
\usepackage{comment}
\usepackage{color}

\renewenvironment{abstract}
{\quotation\small\noindent\rule{\linewidth}{.5pt}\par\smallskip
	{\centering\bfseries\abstractname\par}\medskip}
{\par\noindent\rule{\linewidth}{.5pt}\endquotation}

\newtheorem{The}{Theorem}[section]
\newtheorem{Cor}[The]{Corollary}
\newtheorem{Def}[The]{Definition}
\newtheorem{Lem}[The]{Lemma}
\newtheorem{Rem}[The]{Remark}
\newtheorem{Exam}[The]{Example}
\newtheorem{Prop}[The]{Proposition}
\newtheorem*{Propo*}{Proposition}
\newtheorem{Qst}[The]{Question}
\newcommand{\imu}{\ensuremath{{\bf i}}}
\newcommand{\re}{\ensuremath{\mathbb{R}}}
\newcommand{\co}{\ensuremath{\mathbb{C}}}
\newcommand{\nm}{\ensuremath{\Vert}}

\newcommand{\abf}{\ensuremath{\textbf{a}}}

\newcommand{\bbf}{\ensuremath{\textbf{b}}}

\renewcommand{\Re}{\operatorname{Re}}
\renewcommand{\Im}{\operatorname{Im}}

\newcommand{\cj}[2][3]{%
{}\mkern#1mu\overline{\mkern-#1mu#2}}

\makeatletter
\def\blfootnote{\gdef\@thefnmark{}\@footnotetext}
\makeatother

\title[Mixed maps and contact structures]{Non-degenerate mixed maps and contact structures}
\author{Inácio Rabelo}
\thanks{The first named author was supported by the Coordenação de Aperfeiçoamento de Pessoal de Nível Superior - Brasil (CAPES) - Finance Code 001. The second named author had support from Mexico’s PAPIIT-UNAM project IN101424.}
\address{Instituto de Ciências Matemáticas e Computação, Av. Trabalhador São-Carlense 400, Centro. Caixa Postal: 668 CEP 13560-970, São Carlos SP, Brasil}
\email{rabeloinacio@alumni.usp.br}
\author{José Seade}
\address{Instituto de Matemáticas, Unidad Cuernavaca, Universidad Nacional Autónoma de México, Avenida Universidad s/n, Colonia Lomas de Chamilpa, CP62210, Cuernavaca, Morelos Mexico}
\email{jseade@im.unam.mx}

\numberwithin{equation}{section}
\hypersetup{pdftitle={Non-degenerate mixed maps and contact structures},
pdfauthor={Inacio Rabelo and Jose Seade}}
\date{}

\begin{document}


\maketitle

\begin{center}
\dedicatory{\it \small Dedicated to the memory of Professor Maria Ruas (Cidinha)}
\end{center}

    \begin{abstract}
We develop a theory of mixed isolated complete intersection singularities:
real analytic maps $(\mathbb C^n,0)\to(\mathbb C^k,0)$ in the variables $z$
and $\bar z$ that are submersions along their punctured zero fiber. We give
a Newton non-degeneracy criterion for this property and three sources of
examples: Siegel maps arising from linear actions, mixed coverings, and
mixed Hamm maps. For maps built from blocks of polar degree one that satisfy
a weighted positivity condition, the punctured zero fiber is symplectic and
every spherical link is contact, with the polar action as Reeb flow.
Conversely, contact links of polar weighted maps force symplectic zero
fibers, whereas the zero fibers of Siegel maps are coisotropic. Along the
deformation of a mixed Hamm map to its holomorphic counterpart we prove
exact symplectic rigidity: the punctured zero fibers are identified by
equivariant maps that preserve the weighted Liouville form, the radius and
every coordinate stratum. Hence all their links are strictly contactomorphic
to the holomorphic Hamm link with its canonical contact structure, and the
strict contact type depends only on the polar exponents and the dimensions.
At each fixed radius these identifications extend to contact isotopies
of the ambient weighted sphere. We also construct adapted open books,
show that the ambient open books of
mixed Hamm functions support the standard contact sphere, and deduce
positive-stabilization equivalence in dimension three from Giroux's
correspondence. Finally, we prove contact stability under
higher-order perturbations of a radially homogeneous principal part.

    
    \end{abstract}


\section*{Introduction}

Mixed singularities are defined by real analytic functions of complex
variables and their conjugates. An early explicit example is A'Campo's map
$(u,v,z)\mapsto uv(\bar u+\bar v)+z_1^2+\cdots+z_m^2$ \cite{ACampo1973},
which answered Milnor's question \cite{Milnor1968} on the existence of
non-trivial real singularities with a Milnor fibration. A systematic study
of such singularities began in the mid-1990s, with the aim of using complex
geometry to construct real singularities with rich geometry and topology
\cite{Seade1996,Seade1997}. This led to the twisted Pham--Brieskorn
singularities
$z_1^{a_1}\bar z_{\sigma(1)}+\cdots+z_n^{a_n}\bar z_{\sigma(n)}$, whose
Milnor fibrations are given by $f/|f|$ and define open book decompositions
of the spheres \cite{Ruas2002}, and to the singularities $f\bar g$ studied
in \cite{PichonSeade2008}. This work is the subject of Chapters~VI--VIII
of \cite{Seade2006}, although the name \emph{mixed} was not yet in use.
The fibration theorems for the twisted Pham--Brieskorn singularities rest
on a natural action of $\re_{>0}\times S^1$ for which these functions are
equivariant \cite{Seade1997,Ruas2002}. Cisneros-Molina formalized such
actions as polar $\co^*$-actions and introduced the corresponding class of
polar weighted homogeneous polynomials \cite{CisnerosMolina2008}, showing
that they define Milnor-type fibrations and extending to them Oka's join
theorem for weighted homogeneous complex polynomials. Oka coined the term
\emph{mixed singularities} \cite{Oka2008} and turned the subject into a
broad theory, including the topology of polar weighted homogeneous
hypersurfaces, Newton non-degeneracy and contact structures on links
\cite{Oka2008,Oka2010,Oka2011,Oka2014}; see \cite{Oka2021} for an account.
That theory concerns hypersurfaces. Here we develop the corresponding theory
for complete intersections, building on Oka's work on non-degenerate mixed
functions and contact structures on mixed links.
Our emphasis is on contact geometry. The link of an isolated complex
singularity carries a canonical contact structure \cite{Varchenko1980},
supported by the open books given by Milnor fibrations
\cite{Caubel2006}; this places complex singularities within Giroux's
correspondence between contact structures and open books
\cite{Giroux2002}. We ask to what extent this picture persists for mixed
complete intersections.

A mixed map $F:(\co^n,0)\to(\co^k,0)$ is a \emph{mixed isolated complete
intersection singularity}, or mixed ICIS, if its zero fiber is
positive-dimensional and $F$ has real rank $2k$ at its nonzero points near
the origin. Two features distinguish this setting from the hypersurface
case. Rank conditions must hold simultaneously for $k$ equations and on
coordinate restrictions. And, unlike in the holomorphic case, the
punctured zero fiber need not be symplectic, so its links need not be
contact. Part of the foundational material appears in the first author's
thesis \cite{RabeloThesis2025}.

We first define Newton non-degeneracy for mixed maps, following the
holomorphic theory of Khovanskii \cite{Khovanskii1977}, and prove that
non-degenerate convenient maps are mixed ICIS (Theorem~\ref{icis}); we
relate this to Milnor sets and Milnor fibrations. We then give three
families of examples. \emph{Siegel maps} arise from admissible linear
actions of $\co^k$; their zero fibers parametrize the Siegel leaves and
lead to the LVM-manifolds. \emph{Mixed coverings}, introduced by Oka,
compose holomorphic maps with $w_j\mapsto w_j^{a_j}\bar w_j^{\,b_j}$.
\emph{Mixed Hamm maps} deform the holomorphic complete intersections of
Hamm \cite{Hamm1972}, and carry our contact and rigidity results.

The contact results rest on a simple mechanism. Suppose that the
components of a mixed map $G$ have \emph{polar degree one} for a weighted rotation
$z\mapsto(e^{\imu p_1\theta}z_1,\ldots,e^{\imu p_n\theta}z_n)$, that is,
$G$ is multiplied by $e^{\imu\theta}$. The generator $Y$ of this rotation
is then tangent to the zero fiber, and for the weighted form
$\alpha_{1/\mathbf p}=2\sum_jp_j^{-1}(x_jdy_j-y_jdx_j)$ it satisfies
\[
 \iota_Y\,d\alpha_{1/\mathbf p}=-2\,d\rho,\qquad
 \alpha_{1/\mathbf p}(Y)=2\rho .
\]
If the punctured zero fiber is symplectic for $d\alpha_{1/\mathbf p}$,
these identities show that on each link $Y$ spans the kernel of
$d\alpha_{1/\mathbf p}$, on which $\alpha_{1/\mathbf p}$ is positive: the
link is contact and its Reeb flow is the rotation. Symplecticity is thus
the only issue, and Theorem~\ref{polarcontact} guarantees it when $G$ is
built from functions in disjoint blocks of variables satisfying a
weighted positivity condition on their Jacobians. For rational weights
the links are then Boothby--Wang orbibundles over symplectic orbifolds.
Conversely, for polar weighted maps, contact links force a symplectic
zero fiber (Proposition~\ref{polarconverse}). Siegel maps lie at the other
extreme: their zero fibers are coisotropic, and their links are never
contact for these forms (Proposition~\ref{siegelcoiso}).

Our main result is an exact symplectic rigidity theorem for the mixed Hamm
deformation
\[
 G_t^\ell(z)=\sum_{j=1}^n\lambda_{\ell j}z_j^{a_j}
       \bigl(t+(1-t)|z_j|^{2b_j}\bigr),\qquad 0\le t\le1,
\]
with $a_j\ge2$, $b_j\ge0$, $n\ge k+2$, and all maximal minors of
$\Lambda$ nonzero. Put $V_t=G_t^{-1}(0)$,
\[
 \alpha_{\mathbf a}=2\sum_j a_j(x_j\,dy_j-y_j\,dx_j),
 \qquad \rho=\sum_j|z_j|^2.
\]
Theorem~\ref{exactrigidity} gives equivariant diffeomorphisms
$\Phi_t:V_0\setminus\{0\}\to V_t\setminus\{0\}$ with
\[
 \Phi_t^*(\alpha_{\mathbf a}|_{V_t})=\alpha_{\mathbf a}|_{V_0},
 \qquad \rho\circ\Phi_t=\rho,
\]
which preserve every coordinate stratum and extend homeomorphically
across the origin. Hence, for every radius, the links are \emph{strictly}
contactomorphic: the contact forms, and their Reeb flows with time
parametrization, correspond. The idea is that the Hamm components have
polar degree one for the weights $1/a_j$, which do not depend on $t$, so
all the links share the same Reeb field, namely the polar rotation. A
family of contact forms with a common Reeb field can be identified by a
strict version of Gray stability, not only up to isotopy of the contact
structures. The Liouville field $Z_t$ satisfies $d\rho(Z_t)=\rho$, so its
flow carries one link onto all the others and extends the identification
to the whole punctured fiber.
The common contact type is the canonical one of the holomorphic Hamm
link (Theorem~\ref{comparison}). It depends only on the polar exponents
and the dimensions, not on the radial exponents or on $\Lambda$
(Corollary~\ref{coefficientrigidity}), and the radial monomials may be
replaced by polynomials with nonnegative coefficients
(Corollary~\ref{radialprofiles}).
Moreover, at each fixed radius the strict link identifications extend
to contact isotopies of the ambient weighted sphere
(Corollary~\ref{ambientrigidity}). Thus the rigidity also controls the
contact embedding type of the link, not only its intrinsic contact
structure.

Returning to the picture described at the beginning, these links also
carry adapted open books. They are given by coordinate functions and,
more generally, by block combinations $g$ with $(G,g)$ regular on its
punctured zero fiber (Theorems~\ref{coordbook} and~\ref{milnorbook});
compare \cite{Caubel2006,Koert2005}. For a single mixed Hamm function, the
function itself defines an ambient open book on the sphere, which
supports the standard contact structure (Theorem~\ref{ambientbook}); this
extends the perspective of Oka~\cite{Oka2014}. In dimension three,
Giroux's correspondence then yields common positive stabilizations, both
for these ambient open books and for the open books on the links
identified by our contact isotopies (Corollary~\ref{girouxstable}).

For equal polar exponents, $\alpha_{\mathbf a}$ is a constant multiple of
the standard form, so all these results hold for the standard contact
structure of the sphere. For unequal exponents the standard form is
contact on the ellipsoidal links $V_t\cap\{\sum_j|z_j|^2/a_j=r^2\}$
(Proposition~\ref{ellipsoidal}); its behavior on round spheres is the
question that closes the paper. Finally, we give an exterior-product
criterion for contact links in any codimension (Theorem~\ref{cticis}),
and we prove that higher-order perturbations of a radially homogeneous
map with contact link have contact links of the same type
(Theorem~\ref{radialface}). The latter is a condition on the principal
face functions; equality of Newton boundaries does not suffice.

\section{Mixed ICIS}\label{s1}

    \subsection{Definitions}

    A \textit{mixed map} is a map $F : \co^{n} \to \co^{k}$ which is real analytic in $z = (z_{1}, \dots, z_{n})$ and $\cj{z} = (\cj{z}_{1}, \dots, \cj{z}_{n})$. If $k=1$ we write $f$ and call it a \textit{mixed function}; it has a series expansion $\sum_{\mu,\nu}\lambda_{\mu,\nu}z^{\mu}\cj{z}^{\nu}$. A mixed function has the differentials
    \begin{align*}
        \partial f = \sum_{j=1}^{n}\frac{\partial f}{\partial z_{j}}dz_{j} \quad, \;\;\;
        \cj{\partial} f = \sum_{j=1}^{n}\frac{\partial f}{\partial \cj{z}_{j}}d\cj{z}_{j} \quad , \;\;\; 
        df = \partial f + \cj{\partial}f.
    \end{align*}
    and the complex gradients
    \begin{align}\label{mixgrad}
        Df(z,\cj{z}) = \left(\frac{\partial f}{\partial z_{1}}, \dots, \frac{\partial f}{\partial z_{n}}\right), \; \cj{D}f(z,\cj{z}) = \left(\frac{\partial f}{\partial \cj{z}_{1}}, \dots, \frac{\partial f}{\partial \cj{z}_{n}}\right) \in \co^{n}.
    \end{align}

    A point where the real Jacobian of $F$ does not have maximal rank is a \textit{mixed singularity}. We write $\Sigma_{F}$ for the set of mixed singularities and $V_{F}=F^{-1}(0)$. They are characterized as follows \cite[Proposition 4]{Ribeiro2021}.

    \begin{Prop}\label{chmm}
        Let $F = (f^{1}, \dots, f^{k}) : \co^{n} \longrightarrow \co^{k}$ be a mixed map germ. Then $a \in \co^{n}$ is a mixed singularity if and only if there exists $\alpha_{1}, \dots, \alpha_{k} \in \co$ non-simultaneously vanishing such that 
        \begin{align}\label{eqchmm}
            \alpha_{1}\cj{D}f_{1}(a) + \dots + \alpha_{k}\cj{D}f_{k}(a) = -\cj{\alpha}_{1}\cj{Df_{1}}(a) - \dots - \cj{\alpha}_{k}\cj{Df_{k}}(a).
        \end{align}
    \end{Prop}

    For $k=1$ this is \cite[Proposition 1]{Oka2008}:

    \begin{Cor}\label{okacri}
        Let $f : (\co^{n},0) \longrightarrow (\co,0)$ be a mixed function germ. Then $a \in \co^{n}$ is a mixed singularity if and only if there exists a complex number $\alpha$ such that $\nm \alpha \nm = 1$ and $\cj{D f}(a,\cj{a}) = \alpha \cj{D}f(a,\cj{a})$.
    \end{Cor}

    \begin{Def}\label{dficis}
       Let $F : (\co^{n},0) \longrightarrow (\co^{k},0)$ be a mixed map such that $n > k$ and $V_{F}$ has positive dimension. We say that $F$ is a mixed isolated complete intersection singularity, mixed ICIS for short, if $\Sigma_{F} \cap V_{F} \subseteq \{0\}$.
    \end{Def}

    \begin{Rem}
        \normalfont
        By definition, $V_{F}\setminus\{0\}$ is a smooth manifold of real dimension $2(n-k)$, with no complex structure assumed; the origin may be a regular point. For holomorphic $F$ this is the geometric characterization of isolated complete intersection singularities.
    \end{Rem}

    We now describe two constructions of mixed ICIS; a third, the mixed Hamm maps, appears in Subsection~\ref{s2.2}.

        \subsection{Linear actions on \texorpdfstring{$\co^n$}{Cn}}\label{s1.2}

        A linear vector field $F=\sum_{i=1}^{n} \lambda_{i} z_{i}\,\partial/\partial z_{i}$ on $\co^n$, with all $\lambda_i\ne0$, is in the \emph{Siegel domain} if the convex hull $\mathcal{H}(\lambda_1,\ldots,\lambda_n)$ contains $0$. Its orbits form a one-dimensional holomorphic foliation $\mathcal{F}$, singular only at $0$. A leaf is tangent to the sphere through a point $z\ne0$ exactly when the Hermitian product $\psi(z)=\langle F(z), z\rangle$ vanishes, that is, when
        \begin{equation*}
        \sum_{i=1}^{n} \Re\left(\lambda_{i}\right)\left|z_{i}\right|^{2}=0 \quad \text { and } \quad \sum_{i=1}^{n} \Im\left(\lambda_{i}\right)\left|z_{i}\right|^{2}=0.
        \end{equation*}
        Let $\mathcal{T}_{F}=\psi^{-1}(0)$. If moreover $\lambda_{i} \notin \re\lambda_{j}$ for $i\ne j$ (which, with the Siegel condition, forces $n\ge3$), then by \cite{Camacho1978} the Siegel leaves form an open dense subset of $\co^{n}\setminus\{0\}$; each is an embedded copy of $\co$ meeting $\mathcal{T}_{F}$ exactly once, at its point closest to $0$. Thus $\mathcal{T}_{F} \setminus \{0\}$ is a smooth $(2n-2)$-manifold parametrizing the Siegel leaves. It is transverse to $\mathcal{F}$ and therefore inherits a holomorphic structure \cite{Lopez1997}.

        More generally, consider $k$ commuting linear vector fields
        \begin{equation*}
        F^{j}=\sum_{i=1}^{n} \lambda_{ij} z_{i} \frac{\partial}{\partial z_{i}}, \qquad j=1, \ldots, k,\quad 0<2k<n,
        \end{equation*}
        with $(\lambda_{ij})$ of rank $k$. They generate an action of $\co^k$ whose orbits form a foliation $\mathcal{F}$. Put $\Lambda_{i}=(\lambda_{i1}, \ldots, \lambda_{ik})\in\co^k$. Following \cite{Meersseman2000}:

        \begin{Def}\hfill
            \begin{enumerate}
            \item The action is in the \emph{Siegel domain} if $0 \in \mathcal{H}(\Lambda_{1}, \ldots, \Lambda_{n})$.
            \item It is \emph{admissible} if it is in the Siegel domain and \emph{weakly hyperbolic}: $0 \notin \mathcal{H}(\Lambda_{i_{1}}, \ldots, \Lambda_{i_{2k}})$ for all $i_{1}<\cdots<i_{2k}$. We then call the frame $\mathfrak{F}=(F^{1}, \ldots, F^{k})$ admissible.
        \end{enumerate}
        \end{Def}

        Thus $0$ lies in the convex hull of all the $\Lambda_i$, but not in that of any $2k$ of them. For an admissible frame, the Siegel leaves form a dense open set containing $(\co^*)^n$; each is an embedded copy of $\co^{k}$ with a unique point closest to the origin \cite{Meersseman2000}. These points form the set $\mathcal{T}_{\mathfrak{F}}^{*}=\mathcal{T}_{\mathfrak{F}} \setminus\{0\}$ where $\mathcal{F}$ is tangent to the spheres, and $\mathcal{T}_{\mathfrak{F}}$ is defined by the $k$ mixed equations
        \begin{equation*}
        \psi^{j}(z)=\left\langle F^{j}(z), z\right\rangle=\sum_{i=1}^{n} \lambda_{ij}\left|z_{i}\right|^{2}=0, \qquad j=1, \ldots, k .
        \end{equation*}
        Being transverse to $\mathcal{F}$, the manifold $\mathcal{T}_{\mathfrak{F}}^{*}$ inherits a holomorphic structure \cite{Haefliger1985}, although it is not a complex submanifold of $\co^n$. The following theorem extends to complete intersections the construction of real singularities from complex geometry in \cite{Seade1997,Ruas2002}; it reformulates \cite{Camacho1978,Lopez1997} for $k=1$ and \cite{Meersseman2000} for $k>1$.

        \begin{The}\label{sigicis}
            Let $\mathfrak{F}=(F^{1}, \ldots, F^{k})$ be an admissible frame, and $\psi^{j}$ as above. Then:
            \begin{enumerate}
                \item $\Psi_{\mathfrak{F}}=(\psi^{1}, \ldots, \psi^{k})$ is a mixed ICIS with zero fiber $\mathcal{T}_{\mathfrak{F}}$.
                \item $\mathcal{T}_{\mathfrak{F}}^{*}$ is a complex manifold of dimension $n-k$ parametrizing the Siegel leaves.
            \end{enumerate}
        \end{The}
        \begin{proof}
        For the first assertion, let $z\ne0$ lie in the zero fiber and
        $I=\{i:z_i\ne0\}$. The equation
        $\sum_{i\in I}|z_i|^2\Lambda_i=0$ places zero in the convex hull
        of the active vectors. If these vectors spanned a proper real
        linear subspace of $\co^k\simeq\re^{2k}$, Carath\'eodory's theorem
        in that subspace would express zero as a convex combination of
        at most $2k$ of them, contrary to weak hyperbolicity. They therefore
        span $\re^{2k}$. The independent radial variations of the nonzero
        coordinates show that the real differential of $\Psi_{\mathfrak F}$
        has this same full image. Its only possible critical point on the
        zero fiber is the origin. The parametrization and holomorphic
        structure in the second assertion are the transverse quotient
        constructions of \cite{Lopez1997,Meersseman2000} described above.
        \end{proof}

        \begin{Def}\label{sigicisdef}
            For an admissible frame $\mathfrak{F}$, we call $\Psi_{\mathfrak{F}}$ a \emph{Siegel complete intersection map}.
        \end{Def}

        \begin{Rem}
            \normalfont
            The set $\mathcal{T}_{\mathfrak{F}}$ carries a canonical $\co^*$-action \cite{Lopez1997, Meersseman2000}, which is polar in the sense of \cite{CisnerosMolina2008} and preserves the complex structure of $\mathcal{T}_{\mathfrak{F}}^{*}$. The quotient $\mathcal{T}_{\mathfrak{F}}^{*}/\co^*$ is a compact complex manifold, an LVM-manifold; these are closely related to moment-angle manifolds \cite{Verjovsky2008}.
        \end{Rem}

\begin{Rem}
        \normalfont
        The same construction applies to nonlinear fields. For $F=(\lambda_1 z_{\sigma(1)}^{a_{\sigma(1)}}, \ldots, \lambda_n z_{\sigma(n)}^{a_{\sigma(n)}})$, with $\sigma$ a permutation, the tangency set to the spheres is the zero set of $\langle F(z),z\rangle$; after relabeling and taking $\lambda_i=1$,
        \begin{equation}\label{twist}
        \Psi_{F}=z_{1}^{a_{1}} \bar{z}_{\sigma(1)}+\cdots+z_{n}^{a_{n}} \cj{z}_{\sigma(n)} .
        \end{equation}
        For all $a_{i}\ge2$, these twisted Pham--Brieskorn polynomials have an isolated critical point and a Milnor fibration \cite{Seade1997,Ruas2002}; they were the first mixed functions studied in this way. They carry an action of $\re_{>0} \times S^{1}$, later called a polar action \cite{CisnerosMolina2008}; for $\sigma=\mathrm{id}$ they are Oka's mixed Pham--Brieskorn polynomials \cite{Oka2008}. Systems of such equations lead to the mixed Hamm maps of Subsection~\ref{s2.2}.
 \end{Rem}

        \subsection{Mixed coverings}\label{s1.3}

        Oka's mixed coverings \cite{Oka2014} produce mixed maps from holomorphic ones, for complete intersections as well as hypersurfaces. Let $\abf = (a_{1}, \dots, a_{n})$ and $\bbf = (b_{1}, \dots, b_{n})$ be nonnegative integer vectors with $a_i\neq b_i$ for all $i$. The \textit{mixed covering} is
            \begin{align*}
                \phi_{\abf,\bbf}(w) = \left( w_{1}^{a_{1}}\cj{w}_{1}^{b_{1}}, \dots, w_{n}^{a_{n}}\cj{w}_{n}^{b_{n}}\right);
            \end{align*}
        it is \textit{homogeneous}, written $\phi_{a,b}$, if $a_i=a$ and $b_i=b$ for all $i$. In polar coordinates its factors are $(r,\theta)\mapsto(r^{a_i+b_i},(a_i-b_i)\theta)$, so on $(\co^*)^n$ it is a local real analytic diffeomorphism and a covering of degree $\prod_i|a_i-b_i|$.

\subsection{Real smooth finite determinacy}

The geometric mixed ICIS condition differs from finite determinacy of the
underlying real map. Here $\mathcal K_{\mathbb R}$ denotes the group of
changes of real source coordinates together with multiplication of the
real component vector by invertible matrices of smooth function germs;
for holomorphic maps, holomorphic $\mathcal K$-finite determinacy is a
different condition. See \cite{Wall1981,Ruas1986} for criteria.

\begin{Def}
A mixed ICIS is \emph{algebraically isolated in the real contact
category} if its underlying real map is $C^\infty$-$\mathcal K_{\mathbb R}$-finitely
determined.
\end{Def}

A necessary condition is that the complexification of the real component
map have no critical points on its zero fiber away from the origin: the
contact normal module remains finite-dimensional after complexification,
and its support contains every such point.

\begin{Exam}\label{algsig}
A Siegel map is not algebraically isolated in this sense. Its complexified
real components are
\[
 \sum_i\operatorname{Re}(\lambda_{ij})(\xi_{1,i}^2+\xi_{2,i}^2),\qquad
 \sum_i\operatorname{Im}(\lambda_{ij})(\xi_{1,i}^2+\xi_{2,i}^2),
 \qquad j=1,\ldots,k,
\]
where $z_i=\xi_{1,i}+\imu\xi_{2,i}$ before complexification.
Fix one index $i$, set all other coordinate pairs to zero, and impose
$\xi_{1,i}=\imu\xi_{2,i}$. This gives a complex line in the zero fiber.
At its nonzero points the Jacobian has rank at most one, whereas its target
has dimension $2k\ge2$. Thus the complexified zero fiber has nonisolated
critical points.
\end{Exam}

\begin{Rem}
Even a holomorphic ICIS need not satisfy real smooth contact finite
determinacy. For $f(z_1,z_2)=z_1^2+z_2^2$, the complexification of
$(\operatorname{Re}f,\operatorname{Im}f)$ is equivalent, by linear changes,
to $(u_1^2+u_2^2,v_1^2+v_2^2)$. Its zero fiber has critical points
$(u,v)=(0,v)$ with $v\ne0$ and $v_1^2+v_2^2=0$.
There is no conflict with holomorphic $\mathcal K$-finite determinacy of $f$.
\end{Rem}

\begin{Def}
A monomial $z^\mu\bar z^\nu$ is \emph{purely mixed} if
$\mu_j,\nu_j\ge2$ for at least one index $j$.
A mixed function is purely mixed if all its nonzero monomials are purely mixed.
\end{Def}

\begin{Prop}\label{algicis}
If every component of a mixed map is purely mixed, its underlying real map
is not $C^\infty$-$\mathcal K_{\mathbb R}$-finitely determined.
\end{Prop}
\begin{proof}
Use independent complex variables $u_j,v_j$ for the complexifications of
$z_j,\bar z_j$. Every component and all its first derivatives vanish on
$L=\{u_jv_j=0\text{ for all }j\}$: differentiation of a monomial still
leaves a factor $u_jv_j$ for an index with both exponents at least two.
The same is true of the conjugate components. The set $L$ has positive
complex dimension and lies in the critical zero fiber. Apply the necessary
condition above.
\end{proof}

In particular, suitable mixed coverings give polynomial mixed ICIS that
are not algebraically isolated in the real contact category.

\section{Newton non-degeneracy and mixed complete intersections}\label{s2}

\subsection{Coordinate strata and mixed coverings}

For $\emptyset\ne I\subset\{1,\ldots,n\}$, set
\[
 \co^I=\{z_j=0:j\notin I\},\qquad
 \co^{*I}=\{z\in\co^I:z_j\ne0\text{ for }j\in I\}.
\]
For a convergent mixed series $f=\sum c_{\mu\nu}z^\mu\bar z^\nu$,
its Newton polyhedron is the convex hull of the sets
$\mu+\nu+\re_{\ge0}^n$ for which $c_{\mu\nu}\ne0$. Its Newton boundary
$\Gamma(f)$ is the union of its compact faces. For a positive integral
weight $P$, write $f_P$ for the sum of the monomials of least $P$-degree.
We write $f^I=f|_{\co^I}$. A function is convenient if its Newton polyhedron
meets every coordinate axis. These definitions agree with the usual
holomorphic ones; see \cite{Kouchnirenko1976,Oka2010}.

One may also take the initial forms of the individual components of
$(Df)|_{\co^I}$. Requiring that these have no common zero on $\co^{*I}$,
for every $I$ and $P$, is the partial non-degeneracy condition used in
\cite{Mondal2021}. Identically zero components impose no equation. It implies
that $f$ has no critical points sufficiently near the origin other than,
possibly, the origin. To see this, apply real analytic curve selection to
the critical locus and take the leading terms of the restricted derivatives.
The leading coefficient of the arc would be a common torus zero of their
initial forms. Notice that restricting the ambient gradient differs from
taking the gradient of the restricted function.

\begin{Prop}\label{mixiso}
Let $f:(\co^n,0)\to(\co,0)$ be holomorphic. Suppose that, on a sufficiently
small representative, $d(f^I)$ is nonzero at every point of $\co^{*I}$,
for every nonempty $I$. Then $g=f\circ\phi_{\mathbf a,\mathbf b}$ has no
mixed critical points near the origin outside the origin.
\end{Prop}
\begin{proof}
If $w\in\co^{*I}$, the restriction of the mixed covering to $\co^{*I}$
is a local real diffeomorphism. Thus $g^I=f^I\circ\phi^I$ has real rank
two at $w$. The full differential of $g$ has at least this rank.
\end{proof}

\begin{Cor}
The conclusion holds when $f$ is convenient and Newton non-degenerate.
\end{Cor}
\begin{proof}
For a holomorphic function, positive-weight Euler identities show that
non-degeneracy on the zero set of each face implies that the face function
has no critical points anywhere on its torus. The same holds for coordinate
restrictions, by extending a weight with sufficiently large weights off $I$.
If critical points of $f^I$ accumulated at zero, the leading terms of a
critical arc would contradict this face condition.
\end{proof}

\begin{Exam}\label{countercover}
Partial non-degeneracy of the ambient gradient alone is insufficient in
Proposition~\ref{mixiso}. The function $f=z_1z_2$ satisfies that condition,
but its pullback by $(w_1,w_2)\mapsto(w_1^2\bar w_1,w_2^2\bar w_2)$ is
$w_1^2\bar w_1w_2^2\bar w_2$, whose differential vanishes on both coordinate
axes. The restrictions $f^{\{1\}}$ and $f^{\{2\}}$ are identically zero.
\end{Exam}

\subsection{Non-degenerate maps and convenience}

\begin{Def}\label{ND}
Let $F=(f^1,\ldots,f^k)$ have nonzero components. It is \emph{Newton
non-degenerate} if, for every positive integral weight $P$, the real
differential of $F_P=(f_P^1,\ldots,f_P^k)$ has rank $2k$ at every point
of $F_P^{-1}(0)\cap(\co^*)^n$. It is \emph{strongly non-degenerate} if this
rank condition holds everywhere on $(\co^*)^n$.
\end{Def}

The holomorphic version is the complete-intersection condition of
\cite{Khovanskii1977}; compare \cite{BiviaAusina2007,Morales1984,Nguyen2022}.
An empty face zero fiber imposes no condition for ordinary non-degeneracy.

\begin{Prop}\label{properties}
The following properties hold.
\begin{enumerate}
\item If $F$ is (strongly) non-degenerate and every $f^{j,I}$ is nonzero,
then $F^I$ is (strongly) non-degenerate. This includes $|I|\le k$, with
the same rank-based definition.
\item If $F$ is strongly non-degenerate, the map formed from the nonzero
components of $F^I$ is strongly non-degenerate whenever this list is nonempty.
\item Mixed coverings preserve both kinds of non-degeneracy.
\item Products of (strongly) non-degenerate maps in disjoint sets of
variables are (strongly) non-degenerate.
\end{enumerate}
\end{Prop}
\begin{proof}
For (1), extend a positive weight on $I$ by sufficiently large positive
weights off $I$, simultaneously for all components. Then
$(F^I)_P=F_Q$, regarded as independent of the remaining variables.
A real relation between the differentials on the smaller torus would
therefore give the same relation on the full torus. For (2), perform this
construction only for the nonzero restricted components and extend a rank
relation by zero coefficients for the other components. Strong
non-degeneracy rules it out without requiring a common zero of those other
components. For (3), the radial exponents transform weights by
$P_j=(a_j+b_j)Q_j$, and
$(F\circ\phi)_Q=F_P\circ\phi$. Positive rational weights can be cleared of
denominators. Apply the chain rule on the torus, where $\phi$ is a local
diffeomorphism. Finally, the Jacobian in (4) is block diagonal.
\end{proof}

\begin{Rem}
\normalfont
If $|I|<k$, the torus $\co^{*I}$ has real dimension $2|I|<2k$, so $F^I$ cannot
be strongly non-degenerate. Thus (1) shows that a strongly non-degenerate map
has, for every $I$ with $|I|<k$, some component with $f^{j,I}\equiv0$; this
is why (2) is stated for the nonzero components.
\end{Rem}

\begin{Rem}\label{Dimensions}
A non-degenerate face map in $m\le k$ variables has no torus zero.
For $m<k$ this follows from the rank bound. For $m=k$, at a nonzero zero
$u$ the positive radial Euler vector $(p_j u_j)$ is a nonzero real kernel
vector of $dF_P(u)$, since the $j$th component has Euler derivative
$d_j f_P^j(u)=0$. This again contradicts rank $2k$.
\end{Rem}

\begin{Prop}\label{propertiesconv}
The condition $V_F\cap\co^I=\{0\}$ for every $|I|=k$ is equivalent to
$V_F\cap\co^{*J}=\emptyset$ as germs for every $1\le|J|\le k$.
If $F$ is non-degenerate and all its components are convenient, these
conditions hold.
\end{Prop}
\begin{proof}
The equivalence follows by completing each $J$ to a set of $k$ indices.
For the last assertion, an arc in a nontrivial small-stratum zero fiber
has leading coefficient $u\in\co^{*J}$ and positive weight $P$ with
$(F^J)_P(u)=0$. All the restricted components are nonzero by convenience.
Proposition~\ref{properties} and Remark~\ref{Dimensions} give a contradiction.
\end{proof}

\begin{Def}\label{Conv}
A mixed map $F:(\co^n,0)\to(\co^k,0)$ is \emph{convenient} if
\begin{enumerate}
\item $f^{j,I}\not\equiv0$ for every $j$ and every $|I|>k$;
\item $V_F\cap\co^I=\{0\}$ as germs whenever $|I|=k$.
\end{enumerate}
\end{Def}

By Proposition~\ref{propertiesconv}, a non-degenerate map with convenient
components is convenient. For maps, convenience also concerns the common
zero fiber, not only the separate supports.

\begin{The}\label{icis}
Let $F$ be non-degenerate and convenient, with positive-dimensional zero
fiber. Then $F$ is a mixed ICIS.
\end{The}
\begin{proof}
Suppose that critical points of $F$ on its zero fiber accumulate at zero.
By real analytic curve selection in the incidence set of rank relations,
there are an analytic arc $z(t)$ in this locus and real analytic complex
coefficients $\beta_j(t)$, not all identically zero, such that
\[
 \operatorname{Re}\sum_{j=1}^k\beta_j(t)\,df^j_{z(t)}=0.
\]
The coefficients may be chosen on a unit sphere; multiplying all of them
by a power of $t$ if necessary causes no change to the relation.
Let $I=\{i:z_i(t)\not\equiv0\}$. Convenience gives $|I|>k$ and nonzero
$f^{j,I}$ for every $j$. Write
\[
 z_i(t)=u_it^{p_i}+O(t^{p_i+1}),\qquad
 \beta_j(t)=c_jt^{q_j}+O(t^{q_j+1})
\]
for the nonzero components, and put $d_j=\deg_P(f^{j,I}_P)$ and
$m=\min_{\beta_j\not\equiv0}(q_j+d_j)$.
The equations $f^j(z(t))=0$ give $f^{j,I}_P(u)=0$ for every $j$.
In each real coordinate belonging to index $i\in I$, multiply the
differential relation by $t^{p_i-m}$ and let $t\to0$. One obtains
\[
 \operatorname{Re}\sum_{q_j+d_j=m}c_j\,d(f^{j,I}_P)_u=0
 \quad\text{on }T_u\co^I.
\]
At least one $c_j$ is nonzero. This is a nontrivial real rank relation at
a torus zero of $(F^I)_P$, contrary to Proposition~\ref{properties}.
\end{proof}

\begin{Cor}\label{iciscor}
Under the hypotheses of Theorem~\ref{icis}, every restriction $F^I$ with
$|I|>k$ and positive-dimensional zero fiber is a mixed ICIS.
\end{Cor}
\begin{proof}
Non-degeneracy follows from Proposition~\ref{properties}, and both
convenience conditions pass to coordinate restrictions.
\end{proof}

\begin{Prop}\label{ConvPull}
Mixed coverings preserve convenience.
\end{Prop}
\begin{proof}
They preserve each coordinate stratum, are surjective as germs on it,
and have zero fiber $\{0\}$. Thus $f^{j,I}\circ\phi^I$ is nonzero exactly
when $f^{j,I}$ is nonzero, and a nonzero point of $V_{F\circ\phi}\cap\co^I$
would give a nonzero point of $V_F\cap\co^I$.
\end{proof}

\begin{Def}\label{strongconv}
For a face map $F_P$, let $A(P)$ be the union of the variable indices
actually occurring in its components. We call $F$ \emph{strongly convenient}
if it is convenient and each face map, considered on $\co^{A(P)}$, is
convenient when $|A(P)|>k$, and has zero fiber $\{0\}$ when $|A(P)|\le k$.
\end{Def}

\begin{Prop}\label{iciscor2}
If $F$ is non-degenerate and strongly convenient, a face map with active
variable set $A(P)$, $|A(P)|>k$, and positive-dimensional zero fiber is a
mixed ICIS on $\co^{A(P)}$. Mixed coverings preserve strong convenience.
\end{Prop}
\begin{proof}
An initial form of an initial form is an initial form of the original map:
one can use a sufficiently large multiple of $P$ plus the secondary weight,
extended positively to inactive coordinates. Hence face maps inherit
non-degeneracy. Their convenience is part of the definition, so apply
Theorem~\ref{icis} on the active space. The assertion for coverings follows
from the weight transformation in Proposition~\ref{properties} and the proof
of Proposition~\ref{ConvPull}.
\end{proof}

\begin{Rem}
The active-space qualification is necessary. The face $z_1^2$ of
$z_1^2+z_2^2+z_3^2$ has isolated zero fiber on $\co^{\{1\}}$, but its zero
fiber considered in $\co^3$ is entirely critical. One cannot add inactive
variables and retain the mixed ICIS conclusion.
\end{Rem}

\begin{Prop}\label{disciso}
For a strongly non-degenerate mixed map, every sufficiently small critical
point $p$ satisfies $f^j(p)=0$ for at least one $j$.
\end{Prop}
\begin{proof}
Otherwise select a critical arc on which all components are nonzero.
On its support $I$, every $f^{j,I}$ is nonzero. The leading-term rank
relation used in Theorem~\ref{icis} produces a critical point of a face of
$F^I$. Strong non-degeneracy, which passes to this restriction, excludes it
whether or not it lies on the zero fiber.
\end{proof}

\subsection{Mixed Hamm families}\label{s2.2}

Fix $n>k$, integers $a_j\ge2$, $b_j\ge0$, and a complex $k\times n$
matrix $\Lambda=(\lambda_{\ell j})$ whose $k\times k$ minors are all
nonzero. The bound $a_j\ge2$ makes the origin singular and is used for the
open books below. Set
\begin{align}
 h_{j,t}(z_j)&=z_j^{a_j}A_{j,t}(|z_j|^2),
 &A_{j,t}(s)&=t+(1-t)s^{b_j},\notag\\
 G_t^\ell(z)&=\sum_{j=1}^n\lambda_{\ell j}h_{j,t}(z_j),
 &0&\le t\le1.\label{Hamm def}
\end{align}
When $b_j=0$, $A_{j,t}=1$. Write $F=G_1$, $G=G_0$, and
$V_t=G_t^{-1}(0)$. Thus $F$ is a holomorphic Hamm map \cite{Hamm1972}, and
$G$ is its mixed counterpart. The direct argument below also covers
matrices with some zero entries.

For $s>0$ put
\[
 q_{j,t}(s)=a_j+\frac{2s A'_{j,t}(s)}{A_{j,t}(s)}>0.
\]
For each fixed $t$ it extends smoothly to $s=0$: for $t=0$ it is the
constant $a_j+2b_j$, and for $t>0$ the expression extends directly. (For
$b_j>0$ the extension is not jointly continuous at $(0,0)$; only fixed $t$
is used.)

\begin{Lem}\label{hammregular}
For every $t\in[0,1]$, $V_t\setminus\{0\}$ is smooth of real codimension
$2k$ and transverse to every centered sphere. A nonzero point of $V_t$
has at least $k+1$ nonzero coordinates.
\end{Lem}
\begin{proof}
If the support $I$ of a zero has $|I|\le k$, independence of the corresponding
columns of $\Lambda$ implies $h_{j,t}(z_j)=0$ for every $j\in I$, a
contradiction. For $z_j\ne0$, the map $h_{j,t}$ has angular degree $a_j$
and strictly increasing radial modulus $r^{a_j}A_{j,t}(r^2)$. Its real
differential is invertible. Since $\Lambda_I$ has complex rank $k$, the
real differential of $G_t$ has rank $2k$ at the zero in question.

At this point consider the real radial vector with complex components
\[
 W_j=z_j/q_{j,t}(|z_j|^2)\ (j\in I),\qquad W_j=0\ (j\notin I).
\]
Then $dh_{j,t}(W_j)=h_{j,t}$, so $dG_t(W)=G_t=0$, whereas
$d\rho(W)=2\sum_{j\in I}|z_j|^2/q_{j,t}>0$.
The same vector is tangent to all defining equations and transverse to
the sphere, proving the assertion.
\end{proof}

\begin{The}\label{isofiber}
For every $r>0$, the pairs $(S_r^{2n-1},K_t)$, with $K_t=V_t\cap S_r$,
form a smooth trivial family over $[0,1]$. The ambient trivialization can
be chosen to preserve every coordinate subspace. Its restriction gives
simultaneous diffeomorphisms of all coordinate links with $|I|>k$.
\end{The}
\begin{proof}
The set $\mathcal E=\{(z,t)\in S_r\times[0,1]:G_t(z)=0\}$ is a smooth
compact manifold over the interval; the projection is a submersion by
Lemma~\ref{hammregular}. We construct a lift that also respects coordinate
subspaces. For $\zeta\in\co^n$, regarded as a real vector, define
\[
 L_{z,t}(\zeta)=\bigl(dG_t(z)(z_1\zeta_1,\ldots,z_n\zeta_n),
 d\rho_z(z_1\zeta_1,\ldots,z_n\zeta_n)\bigr).
\]
At a point of $\mathcal E$, the active-coordinate derivatives give rank
$2k$ for the first component, and the vector $W$ in the preceding proof
gives the additional radial direction. Thus $L_{z,t}$ is surjective.
Its Euclidean right inverse $L^*(LL^*)^{-1}$ is smooth on a neighborhood
of $\mathcal E$. Apply it to $(-\partial_tG_t,0)$ and set
$X_j=z_j\zeta_j$. Then
\[
 dG_t(X)+\partial_tG_t=0,\qquad d\rho(X)=0,
 \qquad X_j=0\text{ if }z_j=0.
\]
A cutoff supported in that neighborhood and equal to one near
$\mathcal E$ extends this time-dependent field to the sphere. Its flow
exists on the whole compact interval, preserves all coordinate subspaces,
and maps the links to one another. This proves all the assertions.
\end{proof}

\begin{Rem}
A smooth choice of the same construction depending on $r>0$ also
trivializes the punctured zero fibers on annuli; these diffeomorphisms need
not extend through the origin.
\end{Rem}

\begin{Exam}\label{Exmpstrng}
If all square minors of a $k\times n$ matrix $\Lambda$ are nonzero, the
augmented matrix $[\Lambda\mid I_k]$ has every maximal minor nonzero.
Consequently the maps
\[
 H^\ell(z,w)=\sum_{j=1}^n\lambda_{\ell j}z_j^{a_j}|z_j|^{2b_j}
             +w_\ell^{c_\ell}|w_\ell|^{2d_\ell},
 \qquad a_j,c_\ell\ge2,\quad b_j,d_\ell\ge0,
\]
are mixed Hamm ICIS in $\co^{n+k}$. The maximal-minor assertion follows
by expanding along the chosen identity columns. Individual components
need not be convenient, since they omit the other $w$ variables.
Nevertheless the direct rank proof above applies. For perturbations that
leave all face functions unchanged, non-degeneracy persists; to apply
Theorem~\ref{icis} one must in addition verify its convenience conditions.
\end{Exam}

\subsection{Milnor sets and fibrations}

With $\rho=\sum|z_j|^2$, the Milnor set is the real analytic set germ
\[
 M(F)=\{z:\operatorname{rank}_{\re}d(F,\rho)_z<2k+1\}.
\]
Images and discriminants of real analytic maps need not be well-defined
set germs; we call $F$ \emph{nice} if $\operatorname{Im}F$ and
$\Delta_F=F(\Sigma_F)$ are, as in \cite{Ribeiro2019,Tibar2019}.

\begin{Prop}\label{milnorzero}
If $F$ is a mixed ICIS, then $M(F)\cap V_F\subseteq\{0\}$ as germs. Hence,
whenever $\Delta_F$ is a well-defined germ,
\[
 \overline{M(F)\setminus F^{-1}(\Delta_F)}\cap V_F\subseteq\{0\}.
\]
\end{Prop}
\begin{proof}
On $V_F\setminus\{0\}$ the differential $dF$ is surjective. Thus membership
in $M(F)$ means $d(\rho|_{V_F})=0$. If such points accumulated at zero,
an analytic arc $z(t)$ of them would satisfy
$(\rho(z(t)))'=0$. Its limiting value at zero is zero, contradicting
$z(t)\ne0$. For the second assertion use closedness of the analytic set
$M(F)$.
\end{proof}

\begin{The}\label{Milnorfib}
A nice mixed ICIS admits a Milnor tube fibration over the regular values
in its image. In particular this applies to nice non-degenerate convenient
mixed maps with positive-dimensional zero fiber.
\end{The}
\begin{proof}
For a sufficiently small fixed $r$, Proposition~\ref{milnorzero} gives
transversality of the zero fiber to $S_r$. By compactness, all fibers with
values sufficiently close to zero meet $S_r$ transversely wherever they
are regular. Choose $\delta=\delta(r)>0$ accordingly. Over regular target
values the restriction of $F$ to the closed ball is a proper submersion,
also on its boundary. Ehresmann's theorem with boundary gives, in
particular, the locally trivial fibration
\[
 F:B_r\cap F^{-1}(S_\delta^{2k-1}\setminus\Delta_F)
       \longrightarrow S_\delta^{2k-1}\setminus\Delta_F
\]
over its image. This is the usual Milnor--Hamm tube construction; compare
\cite{Ribeiro2019,CisnerosMolina2023}. The last assertion follows from
Theorem~\ref{icis}.
\end{proof}

More generally, let $\mathcal I_F$ consist of the supports $I$ for which
every $f^{j,I}$ is nonzero. We show now that a leading-term rank argument also yields that
no zero-fiber arc in $M(F)$ can have support in $\mathcal I_F$.

\begin{The}
    Let $F: (\co^{n},0) \longrightarrow (\co^{k},0)$ be a non-degenerate mixed map germ. Then $M(F) \cap V_{F} \cap \co^{*I} = \emptyset$ for every $I \in \mathcal{I}_{F}$.
\end{The}
    \begin{proof}
            The argument is based on the proof of \cite[Lemma 31]{Oka2010}. If this statement is false for some $I \in \mathcal{I}_{F}$, by the characterization of points in the Milnor set presented in \cite[Propositions 5]{Ribeiro2021} and the curve selection lemma, there exists a real analytic curve $w(t) \subset M(F) \cap V_{F} \cap \co^{*I}$ such that $\lim_{t \to 0} w(t) = 0$ and others real analytic curves $\beta_{1}(t), \dots, \beta_{k}(t) \in \co, \lambda(t) \subset \re$ non-simultaneously vanishing such that:
            \begin{align}\label{relation}
                \lambda(t)w(t) = \sum_{j=1}^{k}\beta_{j}(t)\cj{D}f^{j}(w(t), \cj{w}(t)) + \cj{\beta}_{j}(t)\cj{Df^{j}}(w(t), \cj{w}(t)).
            \end{align}
            We may write the analytic expansions of the curves:
            \begin{align*}
                \lambda(t) &= \lambda_{0}t^{s} + o(t),\; \lambda_{0} \in \re\,, s \ge 1, \\
                w_{i}(t) &= a_{i}t^{p_{i}} + o(t),\; a_{i} \in \co^{*}\,, p_{i} \ge 1, \\
                \beta_{j}(t) &= b_{j}t^{q_{j}} + o(t),\; b_{j} \in \co^{*}\,,q_{j} \ge 1,
            \end{align*}
            for all $i \in I$ and $j$ such that $\beta_{j} \not\equiv 0$. If $\lambda(t) \equiv 0$, then we obtain a relation as in Theorem \ref{icis} that contradicts the non-degeneracy of $F^{I}$, provided that $I \in \mathcal{I}_{F}$. Then $\lambda(t) \not\equiv 0$, in particular $\lambda_{0} \neq 0$. For each $i$, the relation then becomes
            {\footnotesize
            \begin{align*}
                a_{i}\lambda_{0}t^{s+p_{i}} + O(t^{s+p_{i}}) = \sum_{l \in L}\left( b_{l}\frac{\partial f^{l,I}}{\partial \cj{z}_{i}} + \cj{b}_{l}\cj{\frac{\partial f^{l,I}}{\partial z_{i}}}\right)t^{q_{l}+d_{l}-p_{i}} + O(t^{q_{l}+d_{l}-p_{i}}),
            \end{align*}}
            \hskip -4pt where $L = \{l : \beta_{l} \not\equiv 0\}$ and $d_{j}$ is the weighted homogeneous degree of $f^{j,I}_{P}$. Let $m = \min_{l}\{ q_{l} + d_{l}\}$. It follows that:
            \begin{align}\label{loworder}
                \sum_{q_{l}+d_{l} = m} b_{l}\frac{\partial f^{l,I}_{P}}{\partial \cj{z}_{i}}(a,\cj{a}) + \cj{b}_{l}\cj{\frac{\partial f^{l,I}_{P}}{\partial z_{i}}}(a,\cj{a}) = 
                \begin{cases}
                    0,&\, \text{if}\; m - p_{i} < s + p_{i}, \\
                    \lambda_{0}a_{i},&\, \text{if}\; m - p_{i} = s+p_{i}.
                \end{cases} 
            \end{align}
           Let us define $K = \{ i : m - p_{i} = s + p_{i}\}$. We claim that $K \neq \emptyset$. Otherwise, since $w(t) \subset V_{F}$, one has that $a \in \co^{*I} \cap \left(F^{I}\right)_{P}^{-1}(0)$ is a critical point of $\left(F^{I}\right)_{P}$. Since $I \in \mathcal{I}_{F}$, we obtain that $F^{I}$ is non-degenerate, which is a contradiction. Consider now the following expression:   
            \begin{align}\label{eqnd}
               \left\langle \sum_{l\in L}\beta_{l}(t)\cj{D}F(w(t),\cj{w}(t))+\cj{\beta}_{l}(t)\cj{DF}(w(t),\cj{w}(t)), w'(t)\right\rangle = \lambda(t) \langle w(t), w'(t) \rangle.
            \end{align}
            The lower degree term on the left-hand side of \eqref{eqnd} is $m-1$, and its coefficient is
            \begin{align*}
                \sum_{i=1}^{n}\sum_{q_{l}+d_{l}=m}\left( b_{l}\frac{\partial f^{l,I}}{\partial \cj{z}_{i}} + \cj{b}_{l}\cj{\frac{\partial f^{l,I}}{\partial z_{i}}}\right)\cj{a}_{i}p_{i}.
            \end{align*}
            By \eqref{loworder}, this sum is equal to
            \begin{align}\label{eqnd2}
                \sum_{i\in K}\sum_{q_{l}+d_{l}=m} \left( b_{l}\frac{\partial f^{l,I}}{\partial \cj{z}_{i}} + \cj{b}_{l}\cj{\frac{\partial f^{l,I}}{\partial z_{i}}}\right)\cj{a}_{i}p_{i} = \lambda_{0}\sum_{i\in K}\nm a_{i}\nm^{2}p_{i} \neq 0,
            \end{align}
            since $K \neq \emptyset$. We may rewrite it as 
            \begin{align*}
                \sum_{q_{l}+d_{l}=m} \langle b_{l}\cj{D}f_{P}^{l,I}(a,\cj{a}), Pa\rangle + \langle \cj{b}_{l}\cj{Df^{l,I}}(a,\cj{a}), Pa\rangle \neq 0,
            \end{align*}
            where $Pa = (p_{1}a_{1}, \dots, p_{n}a_{n})$. Hence, \eqref{eqnd2} implies that 
            \begin{align}\label{eqnd3}
               \Re\left( \langle b_{l_{0}}\cj{D}f_{P}^{l_{0},I}(a,\cj{a}), Pa\rangle + \langle \cj{b}_{l_{0}}\cj{Df^{l_{0},I}_{P}}(a,\cj{a}), Pa\rangle\right) \neq 0
            \end{align}
            for some $l_{0}$. On the other hand, since $w(t) \subset V_{F}$, one has $f^{l_{0}}(w(t),\cj{w}(t)) \equiv 0$, and thus
            \begin{align}\label{eqnd4}
                \frac{d}{dt}f^{l_{0},I}(w(t), \cj{w}(t)) = \left(\langle \cj{Df^{l_{0},I}_{P}}(a,\cj{a}), Pa \rangle + \langle \cj{\cj{D}f^{l_{0},I}_{P}}(a,\cj{a}), P\cj{a}\rangle\right) t^{d_{l_{0}}-1} + O(t^{d_{l_{0}}-1}) \equiv 0.
            \end{align}
            Then, we obtain that:
            \begin{align}\label{eqnd5}
                \langle \cj{Df^{l_{0},I}_{P}}(a,\cj{a}), Pa \rangle + \langle \cj{\cj{D}f^{l_{0},I}_{P}}(a,\cj{a}), P\cj{a}\rangle = 0.
            \end{align}
            If we multiply \eqref{eqnd5} by $\cj{b}_{l_{0}}$, the resulting equation has real part equal to \eqref{eqnd3}, which leads to a contradiction, and we conclude the proof.
    \end{proof}

Thus, for a nice non-degenerate map, it suffices to impose
\[
 \overline{M(F)\setminus F^{-1}(\Delta_F)}\cap V_F\cap\co^I
 \subseteq\{0\}\qquad(I\notin\mathcal I_F)
\]
to obtain the usual Milnor-set criterion for a tube fibration
\cite{Ribeiro2019}; additional hypotheses are needed only on the
strata $I\notin\mathcal I_F$.

\begin{Cor}\label{coricis}
The mixed Hamm map $G_0$ and the holomorphic Hamm map $G_1$ have tube and
sphere Milnor fibrations, smoothly equivalent after restricting the tube
to the interior of the ball.
\end{Cor}
\begin{proof}
For either endpoint, the critical locus consists of points with fewer than
$k$ nonzero coordinates, since the coordinate derivatives vanish at zero
and are invertible elsewhere. Its discriminant is the finite union of the
complex spans of at most $k-1$ columns of $\Lambda$; the coordinate mixed
monomials are surjective as germs. The image is the full target germ, so
the map is nice and the discriminant is radial. The tube assertion follows
from Theorem~\ref{Milnorfib}. At $t=0$ the positive radial weights are
$1/(a_j+2b_j)$, and at $t=1$ they are $1/a_j$; all target components have
the same radial degree one. The corresponding radial Euler field is
transverse to spheres, preserves $G_t/\|G_t\|$, and increases $\|G_t\|$.
The radial weighted homogeneous sphere-fibration and equivalence theorem
of \cite[Theorem 5.3]{Ribeiro2019b} applies (clear denominators in the weights).
\end{proof}

\section{Contact criteria and explicit weighted structures}\label{s3}

\subsection{The ambient form and the exterior criterion}

For the contact results we assume $n\ge k+2$, so the links have dimension
at least three. Set
\begin{equation}\label{naco}
 \rho=\sum_{j=1}^n|z_j|^2,\qquad
 \alpha=2\sum_{j=1}^n(x_j\,dy_j-y_j\,dx_j)
        =-\imu\sum_{j=1}^n(\bar z_j\,dz_j-z_j\,d\bar z_j).
\end{equation}
Then $d\alpha=4\sum_j dx_j\wedge dy_j$. On $S_r$ its Reeb field is
\begin{equation}\label{reebc}
 R_\alpha=\frac1{2r^2}\sum_j(x_j\partial_{y_j}-y_j\partial_{x_j})
 =\frac{\imu}{2r^2}\sum_j(z_j\partial_{z_j}-\bar z_j\partial_{\bar z_j}).
\end{equation}
In general it is not tangent to a mixed link, even when the restriction of
$\alpha$ is a contact form.

Write $f^j=u_j+\imu v_j$ and $dV=dx_1\wedge dy_1\wedge\cdots\wedge
dx_n\wedge dy_n$. Define the real function $\mathcal C_F$ by
\begin{equation}\label{exterior}
 d\rho\wedge\alpha\wedge(d\alpha)^{n-k-1}
 \wedge du_1\wedge dv_1\wedge\cdots\wedge du_k\wedge dv_k
 =\mathcal C_F\,dV.
\end{equation}
The zero fiber is oriented by the ambient orientation and the ordered
real equations $(u_1,v_1,\ldots,u_k,v_k)$; its link has the outward
boundary orientation.

\begin{The}\label{cticis}
For a mixed ICIS and a sufficiently small transverse spherical link $K_F$,
the restriction $\alpha|_{K_F}$ is a contact form if and only if
$\mathcal C_F$ is nowhere zero on $K_F$. It is positive for the specified
orientation if and only if $\mathcal C_F>0$ there.
\end{The}
\begin{proof}
The differentials $d\rho,du_1,dv_1,\ldots,du_k,dv_k$ are independent
on the link. Complete them to local real coordinates. Wedge with these
normal differentials is an isomorphism from top-degree forms on $TK_F$
to ambient top-degree forms at the point. Apply it to
$\alpha\wedge(d\alpha)^{n-k-1}$. Moving the $2k$ equation differentials
past the link form introduces no sign, giving the orientation assertion.
\end{proof}

For $k=1$, the coefficient is a positive dimensional constant times
Oka's expression
\[
 \sum_{i<j}\left(
 |\bar z_i f_{z_j}-\bar z_j f_{z_i}|^2
 -|z_i f_{\bar z_j}-z_j f_{\bar z_i}|^2\right).
\]
For higher codimension, definition~\eqref{exterior} includes all wedge
contributions, including products of $(2,0)$ and $(0,2)$ terms from
different equations.

\begin{Def}
A mixed ICIS is \emph{strictly holomorphic-like relative to $(\rho,\alpha)$}
if $\mathcal C_F>0$ on $V_F\setminus\{0\}$ in some neighborhood of zero.
Strictly anti-holomorphic-like means the corresponding negative inequality.
The condition is imposed on the regular zero fiber, not on the whole
ambient neighborhood; compare \cite{Oka2014}.
\end{Def}

Changing regular defining equations does not change the nonvanishing
contact condition. Their real conormal bases differ by an invertible
matrix along the zero fiber, and the exterior coefficient is multiplied
there by its determinant. Its sign must therefore be compared using the
induced normal orientations. Independence of the radius, up to contact
isotopy, follows from a smooth trivialization of the transverse links and
Gray stability whenever the contact condition holds throughout the radial
interval.

\begin{Exam}\label{coordinatecounter}
Real analytic source diffeomorphisms do not preserve contact with the
fixed ambient form. Let $Q=z_1^2+z_2^2+z_3^2$ and
$\phi(z)=(\bar z_1,z_2,z_3)$. Although $Q$ has its canonical contact link,
$H=Q\circ\phi$ satisfies, at $p=(s,\imu s,0)$ with $s>0$,
\[
 H(p)=0,\qquad \alpha_p=2s(dy_1-dx_2)=-d(\operatorname{Im}H)_p.
\]
Hence $\alpha|_{T_pK_H}=0$. What a diffeomorphism preserves is the
contact condition for the \emph{transported} form and transported link.
\end{Exam}

\begin{Prop}[Fixed-radius contact stability]\label{fixedradius}
Let $F_s$ be a smooth family of real analytic maps on a neighborhood of a
fixed sphere $S_r$. Suppose $F_0|_{S_r}$ is transverse to zero and
$\alpha|_{K_0}$ is contact, where $K_s=F_s^{-1}(0)\cap S_r$.
For sufficiently small $|s|$, the links $K_s$ are smoothly identified
with $K_0$ and their restricted contact structures are isotopic under
these identifications.
\end{Prop}
\begin{proof}
Transversality is open on the compact sphere, and no new zeros occur away
from a neighborhood of $K_0$ for small $s$. The total zero set is a proper
submersion over a smaller parameter interval and is therefore smoothly
trivial. Pull the forms back to $K_0$. Their contact volume is nonzero
at $s=0$ and remains so by compactness and continuity. Gray stability
gives the asserted isotopy.
\end{proof}

\begin{Rem}\label{newtoncaution}
Equality of Newton boundaries alone does not imply a contact or even a
mixed ICIS conclusion. For instance, $Q$ above and
$|z_1|^2+|z_2|^2-|z_3|^2$ have the same mixed Newton boundary, but the
latter has real rank at most one on its positive-dimensional zero fiber.
Preservation of face functions requires the perturbing terms to have
strictly larger degree for every relevant positive weight. Even then,
Theorem~\ref{icis} concerns regularity, and a contact conclusion needs
a separate argument.
\end{Rem}

\subsection{Symplectic diagonal constraints}

For positive real numbers $c_j$, write
\[
 \alpha_{\mathbf c}=2\sum_j c_j(x_j\,dy_j-y_j\,dx_j),\qquad
 \omega_{\mathbf c}=d\alpha_{\mathbf c}=4\sum_j c_j\,dx_j\wedge dy_j.
\]

\begin{Lem}\label{symplecticlemma}
Let $H(z)=\Lambda(h_1(z_1),\ldots,h_n(z_n))$, where each $h_j$ is a
smooth real map $\co\to\co$. At a point $z$ put
$\delta_j=\det_{\re}dh_j(z_j)$. Suppose all $\delta_j\ge0$ and the
columns of $\Lambda$ with $\delta_j>0$ span $\co^k$. Then $dH_z$ is
surjective and its kernel is a symplectic subspace for every
$\omega_{\mathbf c}$ with $c_j>0$.
\end{Lem}
\begin{proof}
Let $D=dH_z$. The skew form on the real target covectors obtained by
applying the ambient Poisson tensor to their pullbacks has matrix
$D\omega_{\mathbf c}^{-1}D^T$. In real dimension two, a real linear
map transforms the standard skew matrix by multiplication by its
determinant. Consequently, up to a fixed nonzero scalar and the
standard real representation, the target skew form is the one associated
to the Hermitian matrix
\[
 B=\Lambda\operatorname{diag}(\delta_j/c_j)\Lambda^*.
\]
For $v\ne0$,
$v^*Bv=\sum_j(\delta_j/c_j)|\Lambda_j^*v|^2>0$ by the spanning
hypothesis. Thus the real skew form is nondegenerate. It follows that
$D$ is onto and that $\ker D$ meets its symplectic orthogonal trivially:
indeed, a vector in the intersection is $\omega_{\mathbf c}^{-1}D^T\eta$
with $D\omega_{\mathbf c}^{-1}D^T\eta=0$, and hence is zero.
\end{proof}

For the functions $h_{j,t}$ in \eqref{Hamm def}, writing $s=|z_j|^2>0$
gives
\begin{equation}\label{jacpositive}
 \det_{\re}dh_{j,t}
 =a_j s^{a_j-1}A_{j,t}(s)
       \bigl(a_jA_{j,t}(s)+2sA'_{j,t}(s)\bigr)>0.
\end{equation}
At $z_j=0$ the determinant is zero, since $a_j\ge2$.
Lemma~\ref{hammregular} and Lemma~\ref{symplecticlemma} imply that
$V_t\setminus\{0\}$ is symplectic for every $\omega_{\mathbf c}$.

\subsection{Weighted contact forms on Hamm links}

\begin{The}\label{weightedcontact}
Let $G_t$ be the mixed Hamm family \eqref{Hamm def}, with $n\ge k+2$.
For every $r>0$ and $t\in[0,1]$, the form
\[
 \beta_t=\alpha_{\mathbf a}|_{K_t},\qquad
 \alpha_{\mathbf a}=2\sum_j a_j(x_j\,dy_j-y_j\,dx_j),
\]
is a contact form. Its Reeb field is
\begin{equation}\label{weightedreeb}
 R_t=\frac1{2r^2}\sum_j\frac1{a_j}
                  (x_j\partial_{y_j}-y_j\partial_{x_j})\bigg|_{K_t}.
\end{equation}
The same assertions hold for every coordinate restriction with more
than $k$ variables (with the nonvanishing one-form convention for
one-dimensional links).
\end{The}
\begin{proof}
Let $Y=\sum_j a_j^{-1}(x_j\partial_{y_j}-y_j\partial_{x_j})$.
The polar identity $dh_{j,t}(Y)=\imu h_{j,t}$ gives
$dG_t(Y)=\imu G_t$, so $Y$ is tangent to $V_t$ and to $K_t$.
Direct calculation yields
\[
 \iota_Y\omega_{\mathbf a}=-2d\rho,
 \qquad \alpha_{\mathbf a}(Y)=2\rho.
\]
Since $V_t\setminus\{0\}$ is symplectic for $\omega_{\mathbf a}$ and
its intersection with $S_r$ is transverse, the kernel of
$\omega_{\mathbf a}|_{TK_t}$ is a line. The displayed identities show
that it is spanned by the nonzero field $Y$, and that the restricted
primitive is nonzero on it. This proves the contact condition and the
normalization \eqref{weightedreeb}. The proofs use only the active
columns and therefore also apply to each coordinate restriction.
\end{proof}

We orient these links by their contact volume. Equivalently, this is
their boundary orientation from the symplectic zero fiber: if $Z_t$ is
the Liouville field of $\alpha_{\mathbf a}|_{V_t}$, then
\[
 \iota_{Z_t}\omega_{\mathbf a}|_{TV_t}=\alpha_{\mathbf a}|_{TV_t},
 \qquad d\rho(Z_t)=\rho>0.
\]
The last equality follows by evaluating the first identity on $Y$ and
using $\omega_{\mathbf a}(Y,\cdot)=-2d\rho$.

\begin{Cor}\label{equalexponents}
If $a_1=\cdots=a_n=a$, the standard form $\alpha|_{K_t}$ is contact
for every $t\in[0,1]$ and every $r>0$.
\end{Cor}
\begin{proof}
Here $\alpha_{\mathbf a}=a\alpha$.
\end{proof}

\begin{Rem}
The weights in Theorem~\ref{weightedcontact} are the polar exponents
$a_j$, not the radial degrees $a_j+2b_j$. For unequal $a_j$ the theorem
concerns the weighted form $\alpha_{\mathbf a}$ rather than the standard
form $\alpha$; see the question at the end of the paper.
\end{Rem}

\subsection{A polar--symplectic contact theorem}\label{s3.4}

The proof of Theorem~\ref{weightedcontact} uses only two features of the
Hamm family: a one-parameter action under which the defining functions have degree
one, and the symplecticity of the regular zero fiber. We now isolate this
principle.

Let $\mathbf p=(p_1,\ldots,p_n)$ be positive real numbers, the \emph{polar
weights}, and let
\[
 Y_{\mathbf p}=\sum_jp_j(x_j\partial_{y_j}-y_j\partial_{x_j}),
\]
the generator of the $\re$-action
$\Phi_\theta(z)=(e^{\imu p_1\theta}z_1,\ldots,e^{\imu p_n\theta}z_n)$.
A mixed function $h$ has \emph{polar degree one} if
$h(\Phi_\theta(z))=e^{\imu\theta}h(z)$ for all
$\theta$, that is, $dh(Y_{\mathbf p})=\imu h$. We use the weighted forms of
Section~\ref{s3} with $c_j=1/p_j$; a direct computation gives
\begin{equation}\label{polaridentities}
 \iota_{Y_{\mathbf p}}\omega_{1/\mathbf p}=-2\,d\rho,\qquad
 \alpha_{1/\mathbf p}(Y_{\mathbf p})=2\rho .
\end{equation}

Partition the variables into disjoint blocks $w^{(1)},\ldots,w^{(N)}$, and
let $h_i(w^{(i)})$ be a mixed function of the variables in the $i$-th block.
Put
\[
 \delta_i=\sum_{j\in\text{block }i}p_j
 \Bigl(\Bigl|\frac{\partial h_i}{\partial z_j}\Bigr|^2
 -\Bigl|\frac{\partial h_i}{\partial\bar z_j}\Bigr|^2\Bigr).
\]
We say that $h_i$ is \emph{weighted holomorphic-like} if $\delta_i>0$ at every
point where $dh_i\neq0$.

\begin{Lem}\label{blocklemma}
Let $H=\Lambda(h_1,\ldots,h_N)$ with $\Lambda\in M_{k\times N}(\co)$, and
suppose that at a point $z$ all $\delta_i\ge0$ and the columns of $\Lambda$
with $\delta_i>0$ span $\co^k$. Then $dH_z$ is surjective and its kernel is
a symplectic subspace for $\omega_{1/\mathbf p}$.
\end{Lem}

\begin{proof}
This is the proof of Lemma~\ref{symplecticlemma}. The Poisson bivector of
$\omega_{1/\mathbf p}$ is a positive multiple of
$\sum_jp_j\,\partial_{x_j}\wedge\partial_{y_j}$. The push-forward of
$\partial_{x_j}\wedge\partial_{y_j}$ by $dh_i$ is the Jacobian determinant
$|\partial h_i/\partial z_j|^2-|\partial h_i/\partial\bar z_j|^2$ times the
standard bivector of $\co$, so the push-forward of the block part is $\delta_i$
times the standard bivector. Thus the target skew form corresponds to the
Hermitian matrix $\Lambda\operatorname{diag}(\delta_i)\Lambda^*$, which is
positive definite, and the argument concludes as before.
\end{proof}

\begin{The}\label{polarcontact}
Let $G=\Lambda(h_1,\ldots,h_N):\co^n\to\co^k$, $n\ge k+2$, where the
$h_i$ are mixed functions in disjoint blocks of variables, each of polar
degree one and weighted holomorphic-like for the weights $\mathbf p$. Let
$U$ be a ball centered at $0$, possibly $U=\co^n$, and suppose that $G$ is
regular at every point of $V_G\cap U\setminus\{0\}$. Then:
\begin{enumerate}
\item $V_G\cap U\setminus\{0\}$ is a symplectic submanifold for
$\omega_{1/\mathbf p}$. It is transverse to every sphere $S_r\subset U$:
the Liouville field $Z$ of $\alpha_{1/\mathbf p}|_{V_G}$ satisfies
$d\rho(Z)=\rho$.
\item For every such $r$, the restriction of $\alpha_{1/\mathbf p}$ to the
link $K_r=V_G\cap S_r$ is a contact form, with Reeb field
$R=Y_{\mathbf p}/(2r^2)$.
\item If the weights $p_j$ are rational, the Reeb flow is periodic. Thus
$(K_r,\alpha_{1/\mathbf p})$ is a quasi-regular contact manifold, the total
space of a Boothby--Wang orbibundle over the symplectic orbifold
$K_r/S^1$, whose symplectic form is induced by $\omega_{1/\mathbf p}$.
\end{enumerate}
\end{The}

\begin{proof}
At a point of $V_G\setminus\{0\}$, $dG=\sum_i\Lambda_i\,dh_i$ is
surjective. If $dh_i\neq0$ then $\delta_i>0$, so the columns with
$\delta_i>0$ span $\co^k$, and Lemma~\ref{blocklemma} gives (1), except
for transversality. Since every $h_i$ has polar degree one,
$dG(Y_{\mathbf p})=\imu G$, so $Y_{\mathbf p}$ is tangent to $V_G$ and to
the spheres. Evaluating $\iota_Z\omega_{1/\mathbf p}=\alpha_{1/\mathbf p}$ on
$Y_{\mathbf p}$ and using \eqref{polaridentities} gives
$2\,d\rho(Z)=2\rho$, which proves transversality. Assertion (2) follows from
\eqref{polaridentities} exactly as in the proof of
Theorem~\ref{weightedcontact}: the kernel of $\omega_{1/\mathbf p}|_{TK_r}$
is a line, spanned by $Y_{\mathbf p}$, on which the form is $2r^2$. For (3),
$Y_{\mathbf p}$ has no zeros on $K_r$, and its flow is periodic when the
weights are rational. Choose $L>0$ so that the flow of $Y_{\mathbf p}$
has period $2\pi L$ and the resulting circle action is effective. Its
infinitesimal generator is $L Y_{\mathbf p}$, so
$\eta=\alpha_{1/\mathbf p}|_{K_r}/(2Lr^2)$ is an invariant connection form.
The restriction of $\omega_{1/\mathbf p}$ is basic and descends to a
symplectic form on the quotient orbifold; the curvature $d\eta$ descends
to its constant rescaling by $1/(2Lr^2)$. This is the orbifold
Boothby--Wang construction \cite{BoothbyWang1958,KegelLange2021}.
\end{proof}

The proof of assertions (1) and (2) also applies when $n=k+1$:
the link is then one-dimensional and the restricted form is nowhere
vanishing. We use this convention for one-dimensional contact bindings.

The converse holds for every polar weighted map.

\begin{Prop}\label{polarconverse}
Let $G$ be a mixed map whose components have polar degree one for the
weights $\mathbf p$, regular on $V_G\setminus\{0\}$ near $0$, with zero
fiber transverse to small spheres. If $\alpha_{1/\mathbf p}|_{K_r}$ is a
contact form for all small $r>0$, then $V_G\setminus\{0\}$ is symplectic for
$\omega_{1/\mathbf p}$ near $0$.
\end{Prop}

\begin{proof}
Let $x\in K_r$ and let $N\in T_xV_G$ be transverse to $S_r$. Write
$\omega=\omega_{1/\mathbf p}$. The field $Y_{\mathbf p}$ is tangent to $K_r$,
and by \eqref{polaridentities} it lies in the kernel of $\omega|_{TK_r}$, which
is a line because $K_r$ is contact; so this kernel is
$\re Y_{\mathbf p}$. Let $v=u+sN$, $u\in T_xK_r$, be in the kernel of
$\omega|_{TV_G}$. Then
$0=\omega(v,Y_{\mathbf p})=2\,d\rho(u)+2s\,d\rho(N)=2s\,d\rho(N)$, so $s=0$.
Hence $u=\lambda Y_{\mathbf p}$, and
$0=\omega(u,N)=-2\lambda\,d\rho(N)$ gives $\lambda=0$.
\end{proof}

\begin{Exam}\label{polarexamples}
The following blocks satisfy the hypotheses of Theorem~\ref{polarcontact}.
\begin{enumerate}
\item The Hamm blocks $h_{j,t}(z_j)$ of \eqref{Hamm def}, with
$p_j=1/a_j$; then $\delta_j=p_j\det_{\re}dh_{j,t}>0$ for $z_j\ne0$ by
\eqref{jacpositive}, and Theorem~\ref{weightedcontact} is the case of one
variable per block.
\item The mixed Brieskorn blocks $z_j^{a}\bar z_j^{\,b}$ with $a>b\ge0$, with
$p_j=1/(a-b)$; here $\delta_j=p_j(a^2-b^2)|z_j|^{2(a+b-1)}>0$ for $z_j\ne0$.
Since $z_j^a\bar z_j^{\,b}=z_j^{a-b}|z_j|^{2b}$, these have the Hamm
form, allowing polar exponent one as well.
\item Holomorphic weighted homogeneous polynomials $h_i(w^{(i)})$ of degree
one for the weights $\mathbf p$; here
$\delta_i=\sum_jp_j|\partial h_i/\partial z_j|^2$, positive wherever
$dh_i\ne0$.
\end{enumerate}
Combining such blocks in a map $G=\Lambda(h_1,\ldots,h_N)$ gives mixed
complete intersections of Thom--Sebastiani--Hamm type provided the zero
fiber is positive-dimensional and $G$ is regular on its punctured zero
fiber. Their links are contact, and are quasi-regular when the polar
weights are rational. For $k=1$ and a single block, the hypothesis is a weighted
version of the holomorphic-like condition of Oka \cite{Oka2014}, and the
conclusion holds for every radius.
\end{Exam}

\subsection{Siegel maps: coisotropic zero fibers}\label{s3.5}

The Siegel maps of Section~\ref{s1.2} lie at the opposite extreme.

\begin{Prop}\label{siegelcoiso}
Let $\Psi_{\mathfrak F}$ be a Siegel complete intersection map. For every
choice of positive weights $c_j$, the manifold $\mathcal T_{\mathfrak F}^*$
is coisotropic for $\omega_{\mathbf c}$, and the kernel of
$\omega_{\mathbf c}$ restricted to the tangent space of its spherical link
has dimension exactly $2k+1$. In particular, no form $\alpha_{\mathbf c}$
restricts to a contact form on the link of a Siegel map.
\end{Prop}

\begin{proof}
The real and imaginary parts of the components $\psi^j=\sum_i\lambda_{ij}
|z_i|^2$ are linear combinations of the functions $|z_i|^2$. Their
Hamiltonian vector fields for $\omega_{\mathbf c}$ are therefore linear
combinations of the rotations $x_i\partial_{y_i}-y_i\partial_{x_i}$. These
rotations preserve $\mathcal T_{\mathfrak F}$ and the spheres, so the $2k$
Hamiltonian fields are tangent to the link. They are linearly independent at
points of $\mathcal T^*_{\mathfrak F}$, because $\Psi_{\mathfrak F}$ is a
submersion there. If $X_f$ is one of them and $v$ is tangent to
$\mathcal T^*_{\mathfrak F}$, then $\omega_{\mathbf c}(X_f,v)=\pm df(v)=0$. Thus
the $2k$ fields lie in the kernel of $\omega_{\mathbf c}$ on
$T\mathcal T^*_{\mathfrak F}$, which proves that $\mathcal T^*_{\mathfrak F}$
is coisotropic. To compute the kernel on the link, include the function
$\rho$. Its differential is independent of the $2k$ defining differentials:
the radial Euler field is tangent to $\mathcal T^*_{\mathfrak F}$ and has
$d\rho(E)=2\rho>0$. All $2k+1$ functions Poisson commute, so their independent
Hamiltonian fields are tangent to the common level set $K$ and span
$(TK)^{\omega_{\mathbf c}}$. Thus $K$ is coisotropic and
$\dim\ker(\omega_{\mathbf c}|_{TK})=2k+1$. A contact form requires this
kernel to have dimension one, whereas $2k+1\ge3$.
\end{proof}

The characteristic foliation of $\mathcal T^*_{\mathfrak F}$ is given by the
orbits of the action of $\re^{2k}$ by the rotations with weights
$\operatorname{Re}\lambda_{ij}/c_i$ and
$\operatorname{Im}\lambda_{ij}/c_i$ (up to a common nonzero factor), which
need not be closed. In the terminology of Theorem~\ref{polarcontact}, the
components of a Siegel map have polar degree zero. Theorem~\ref{polarcontact},
Proposition~\ref{polarconverse} and Proposition~\ref{siegelcoiso} thus place
the two families of mixed complete intersections of this paper at the two
extremes: the Hamm families have symplectic zero fibers and contact links,
while the Siegel maps have coisotropic zero fibers.

\subsection{A Newton criterion: perturbations of the radial face}\label{s3.6}

We now give a criterion in the spirit of Section~\ref{s2}: the contact
property of the link is determined by the initial part of the map for the
weight $P=(1,\ldots,1)$, provided this initial part has a contact link.
A mixed map $H=(H_1,\ldots,H_k)$ is \emph{radially homogeneous} of degrees
$d_1,\ldots,d_k$ if $H_j(sz)=s^{d_j}H_j(z)$ for every real $s>0$; equivalently,
every monomial $z^\mu\bar z^\nu$ of $H_j$ has $|\mu|+|\nu|=d_j$. Then $V_H$ is
a real cone, so if $H$ is regular on $V_H\setminus\{0\}$, its zero fiber is
transverse to every sphere and all its links are diffeomorphic.

\begin{The}\label{radialface}
Let $H:\co^n\to\co^k$, $n\ge k+2$, be a radially homogeneous mixed map of
degrees $d_1,\ldots,d_k$, regular on $V_H\setminus\{0\}$, such that $\alpha$
restricts to a contact form on the nonempty link $K(H)=V_H\cap S_1$. Let
$F=H+R$ be a mixed map germ such that every monomial of $R_j$ has radial
degree $|\mu|+|\nu|>d_j$. Then there is $r_0>0$ such that:
\begin{enumerate}
\item $F$ is a mixed ICIS, and $V_F$ is transverse to $S_r$ for
$0<r<r_0$;
\item $\alpha$ restricts to a contact form on $K_r(F)=V_F\cap S_r$ for
$0<r<r_0$;
\item $(K_r(F),\ker\alpha)$ is contactomorphic to $(K(H),\ker\alpha)$.
\end{enumerate}
\end{The}

\begin{proof}
For $0<s\le s_0$ put $F_s(w)=(s^{-d_1}F_1(sw),\ldots,s^{-d_k}F_k(sw))$, and
$F_0=H$. Each component is $H_j(w)$ plus a convergent series in which every
monomial carries a positive integral power of $s$. Hence $(s,w)\mapsto
F_s(w)$ is real analytic on $[0,s_0]\times\{|w|\le2\}$ for small $s_0$, and
$F_s\to H$ in the $C^1$ topology as $s\to0$.

The radial Euler field is tangent to $V_H$ and satisfies $d\rho(E)=2\rho>0$.
Consequently $d(H,\rho)$ has real rank $2k+1$ on $V_H$ in the compact
annulus $\{1/2\le|w|\le2\}$. Compactness and $C^1$ convergence imply the
same rank condition for $d(F_s,\rho)$ at the zeros of $F_s$ in this annulus
when $s$ is small. Indeed, a sequence of exceptional zeros would converge
to a zero of $H$ where the rank condition persists in a neighborhood.
Since $V_F\cap\{s/2\le|z|\le2s\}$ is the image of
$V_{F_s}\cap\{1/2\le|w|\le2\}$ under $w\mapsto sw$, $F$ is regular on
$V_F\setminus\{0\}$ near the origin, and its zero fiber is transverse to
all sufficiently small spheres. The nonempty regular link of $H$ persists
under this perturbation, so the zero fiber of $F$ is positive-dimensional.
This proves (1).

The dilation $w\mapsto sw$ maps $K_1(F_s)$ onto $K_s(F)$, and it pulls $\alpha$
back to $s^2\alpha$. Thus (2) and (3) for $r=s$ are equivalent to the
corresponding statements for the links $K_1(F_s)$ of the fixed sphere $S_1$
with the form $\alpha$. These follow from Proposition~\ref{fixedradius}
applied to the family $F_s$ near $s=0$, whose link at $s=0$ is $K(H)$.
\end{proof}

\begin{Cor}\label{radialfacecor}
Let $F=H+R$ be as in Theorem~\ref{radialface}, where either
\begin{enumerate}
\item $H$ is a holomorphic map with homogeneous components defining an
isolated complete intersection singularity, or
\item $H$ is the endpoint $G_0$ of a mixed Hamm family with equal exponents
$a_j=a$ and $b_j=b$.
\end{enumerate}
Then for small $r$ the links of $F$ are contact submanifolds of the standard
sphere, contactomorphic to the canonical contact link of the holomorphic
complete intersection $H$ in case (1), and of $G_1$ in case (2).
\end{Cor}

\begin{proof}
In case (1) the link of $H$ is the boundary of a strictly pseudoconvex
domain in the cone $V_H$, hence contact for $\alpha=-d^c\rho$, with its
canonical structure \cite{Varchenko1980}. In case (2), $G_0$ is radially
homogeneous of degree $a+2b$, and Corollary~\ref{equalexponents} and
Theorem~\ref{comparison} give the contact property and its type.
\end{proof}

In terms of Newton polyhedra, the hypothesis says that the face function
$F_P$ of each component for $P=(1,\ldots,1)$ is the corresponding component
of $H$, and that $H$ has a contact link. This is much stronger than equality
of Newton boundaries, which does not suffice (Remark~\ref{newtoncaution}).
For other weights the dilation $z\mapsto(s^{q_1}z_1,\ldots,s^{q_n}z_n)$
preserves neither the spheres nor the form $\alpha$, and the argument does
not apply. This is the same phenomenon as in the question at the end of
the paper for unequal polar exponents.

\section{Coordinate open books and comparison with holomorphic links}\label{s4}

\subsection{A Reeb criterion and coordinate open books}

An open book on a closed oriented manifold $K$ is a pair $(B,\theta)$,
where $B$ is a codimension-two submanifold with trivial normal bundle and
$\theta:K\setminus B\to S^1$ is a locally trivial fibration agreeing
with the normal angular coordinate near $B$. A contact form is adapted
if its restriction to $B$ is a positive contact form and its differential
is symplectic on the pages with boundary orientation agreeing with that
of $B$. We use these conventions as in \cite{Caubel2006,Giroux2002}.

If $\beta$ is contact, $d\beta$ is symplectic on a hypersurface of $K$
exactly when its Reeb field is transverse to that hypersurface. In fact,
the quotient $TK/\re R_\beta$ is symplectic, and restriction of the quotient
map to a transverse hyperplane is an isomorphism. Thus positivity of
$d\theta(R_\beta)$ supplies the page condition. Near the binding one must
also verify its contact condition and its orientation.

\begin{The}\label{coordbook}
For the family in Theorem~\ref{weightedcontact} and every $j$, set
\[
 B_{j,t}=K_t\cap\{z_j=0\},\qquad
 \theta_j=\frac{z_j}{|z_j|}:K_t\setminus B_{j,t}\longrightarrow S^1.
\]
Then $(B_{j,t},\theta_j)$ is an open book adapted to $\beta_t$.
The page symplectic form is $d\beta_t$, and
\[
 d(\arg z_j)(R_t)=\frac1{2r^2a_j}>0.
\]
\end{The}
\begin{proof}
The binding is the link of the coordinate restriction of $G_t$ and is
smooth by Lemma~\ref{hammregular}. At $z_j=0$, $dh_{j,t}=0$ because
$a_j\ge2$. The real coordinate vectors in the $z_j$-plane are therefore
tangent to $K_t$, and $z_j|_{K_t}$ is a submersion along the binding.
The resulting complex normal coordinate trivializes its normal bundle.
Compactness gives a product neighborhood in which $\theta_j$ is the
normal angular coordinate.

The polar field $Y$ is tangent to $K_t$, preserves $|z_j|$, and satisfies
$d(\arg z_j)(Y)=1/a_j$. Its complete flow on the compact link preserves
the complement of the binding and supplies local trivializations of
$\theta_j$ over arcs of $S^1$. Hence it is an open book. The displayed
Reeb calculation and the quotient argument above prove that $d\beta_t$
is symplectic on every page. The restricted form on the binding is contact
by Theorem~\ref{weightedcontact} in the smaller coordinate space.

For the orientation assertion, at a point of the binding the tangent
space of the symplectic zero fiber splits as that of the restricted zero
fiber plus the $z_j$-plane. The latter has symplectic form
$4a_j\,dx_j\wedge dy_j$, and the radial constraint has no component in
that plane. Thus the induced normal orientation of the binding is the
complex orientation of $z_j$. With the positive angular direction just
computed, the contact orientation on the binding is the boundary
orientation of the symplectic pages. This verifies the remaining
adaptation condition.
\end{proof}

When $k=1$ and $b_j=0$, the coordinate
construction is closely related to the Brieskorn open books of
\cite{Koert2005}; \cite{Caubel2006} treats general holomorphic Milnor open books.

\subsection{Milnor open books}

Coordinate functions are a special choice. In the holomorphic theory, the
Milnor open book of an additional function $g$ on an ICIS link supports the
canonical contact structure \cite{Caubel2006}. The polar structure gives the
analogous statement in the mixed setting.

\begin{The}\label{milnorbook}
Let $G=\Lambda(h_1,\ldots,h_N)$ and $U$ be as in
Theorem~\ref{polarcontact}, and let $g=\sum_i\mu_ih_i$ with
$\mu\in\co^N$. Suppose that the map $G'=(G,g)$ is regular at every point of
$V_{G'}\cap U\setminus\{0\}$. Then, for every sphere $S_r\subset U$,
\[
 B_r=K_r\cap\{g=0\},\qquad \theta=\frac g{|g|}:K_r\setminus B_r\to S^1,
\]
is an open book on $K_r$ adapted to $\alpha_{1/\mathbf p}|_{K_r}$, and
$d(\arg g)(R)=1/(2r^2)$.
\end{The}

\begin{proof}
The matrix of $G'$ is $\left[\begin{smallmatrix}\Lambda\\
\mu\end{smallmatrix}\right]$, so Theorem~\ref{polarcontact}, including the one-dimensional
case explained immediately after its proof, applies to $G'$:
the binding $B_r$ is a smooth link, transverse to $S_r$, on which
$\alpha_{1/\mathbf p}$ restricts to a contact form (a nowhere vanishing
$1$-form if $\dim B_r=1$). Since $G'$ is regular along $B_r$ and its zero
fiber is transverse to $S_r$, the restriction of $g$ to $K_r$ is a submersion
near $B_r$; it is a complex normal coordinate, and $\theta$ is the normal
angular coordinate. Since $dg(Y_{\mathbf p})=\imu g$, we have
$d(\arg g)(Y_{\mathbf p})=1$, so the flow of $Y_{\mathbf p}$ preserves $|g|$,
rotates $\theta$ at unit speed, and gives local trivializations of $\theta$.
By Theorem~\ref{polarcontact}, $d(\arg g)(R)=1/(2r^2)>0$, and the Reeb
criterion above shows that $d\alpha_{1/\mathbf p}$ is symplectic on the pages.

It remains to compare orientations along the binding. Let $x\in B_r$, write
$\omega=\omega_{1/\mathbf p}$, and let $N$ be the $\omega$-orthogonal of
$T_xV_{G'}$ in $T_xV_G$. It is a symplectic plane, and $dg|_N:N\to\co$ is an
isomorphism. The Poisson bracket of $\operatorname{Re}g$ and
$\operatorname{Im}g$ on $V_G$ is the Dirac bracket with respect to the real
components of $G$, which corresponds to the Schur complement of the
$G$-block in the positive definite Hermitian matrix
$\left[\begin{smallmatrix}\Lambda\\ \mu\end{smallmatrix}\right]
\operatorname{diag}(\delta_i)
\left[\begin{smallmatrix}\Lambda\\ \mu\end{smallmatrix}\right]^*$ of
Lemma~\ref{blocklemma}. This Schur complement is positive. Since
$T_xV_{G'}\subset\ker dg$, the bracket is computed on $N$, and therefore
$dg|_N$ carries the symplectic orientation of $N$ to the complex orientation
of $\co$. As in the proof of Theorem~\ref{coordbook}, this shows that the
contact orientation of the binding is the boundary orientation of the
pages.
\end{proof}

\begin{Cor}\label{hammmilnorbook}
Let $G_t$ be the mixed Hamm family and let $\mu\in\co^n$ be such that every
maximal minor of $\left[\begin{smallmatrix}\Lambda\\ \mu\end{smallmatrix}
\right]$ is nonzero, and put $g_t=\sum_j\mu_jh_{j,t}(z_j)$. For every $t\in[0,1]$
and $r>0$, $(K_t\cap\{g_t=0\},g_t/|g_t|)$ is an open book adapted to
$\beta_t$. For $t=1$ it is the Milnor open book of $g_1$ on the holomorphic
Hamm link.
\end{Cor}

\begin{proof}
By Lemma~\ref{hammregular} with $k+1$ rows, $(G_t,g_t)$ is regular on its
zero fiber outside the origin, and the Hamm blocks satisfy the hypotheses of
Theorem~\ref{polarcontact} with $p_j=1/a_j$
(Example~\ref{polarexamples}). Apply Theorem~\ref{milnorbook}.
\end{proof}

The coordinate open books of Theorem~\ref{coordbook} are not of this form:
$h_{j,t}$ has a critical point on $\{z_j=0\}$, and $h_{j,t}/|h_{j,t}|$ is the
$a_j$-th power of the coordinate angle.

\subsection{Ambient open books for mixed functions}\label{ambientgiroux}

For a mixed function $f:\co^n\to\co$, its link $K_f$ can play two
different roles: it is the binding of an open book on $S_r^{2n-1}$,
and, when its dimension is at least three, it can itself carry the
open books of Theorems~\ref{coordbook} and~\ref{milnorbook}.
The following statement concerns the first construction.

\begin{The}[Ambient contact compatibility]\label{ambientbook}
Let $n\ge2$, $a_j\ge2$, $b_j\ge0$, and $\lambda_j\in\co^*$, and set
\[
 f_t(z)=\sum_{j=1}^n\lambda_jz_j^{a_j}
                  \bigl(t+(1-t)|z_j|^{2b_j}\bigr),\qquad 0\le t\le1.
\]
For every $r>0$, put $K_{t,r}=f_t^{-1}(0)\cap S_r^{2n-1}$ and
$\vartheta_t=f_t/|f_t|$. Then $(K_{t,r},\vartheta_t)$ is an open book
on $S_r^{2n-1}$ adapted to the fixed form
\[
 \beta=\alpha_{\mathbf a}|_{S_r^{2n-1}}.
\]
Its pages have exact symplectic form $d\beta$ on their interiors, and
\[
 d(\arg f_t)(R_\beta)=\frac1{2r^2}.
\]
The supported contact structure $\ker\beta$ is isotopic to the standard
contact structure of the sphere.
\end{The}

\begin{proof}
Write $Y=\sum_j a_j^{-1}(x_j\partial_{y_j}-y_j\partial_{x_j})$.
The identities $\iota_Yd\alpha_{\mathbf a}=-2d\rho$ and
$\alpha_{\mathbf a}(Y)=2\rho$ show that $R_\beta=Y/(2r^2)$.
Indeed, the radial Liouville field
$\frac12\sum_j(x_j\partial_{x_j}+y_j\partial_{y_j})$ is transverse
to the sphere, so $\beta$ is contact.

The regularity and transversality arguments of Lemma~\ref{hammregular}
apply for $k=1$, including $n=2$. Thus $f_t|_{S_r}$ is a submersion
along $K_{t,r}$, and supplies a trivial normal disk coordinate there.
Polar homogeneity gives $df_t(Y)=\imu f_t$. The complete flow of $Y$
preserves $|f_t|$ and rotates its argument at unit speed. It therefore
trivializes $\vartheta_t$ over arcs of the circle; together with the
normal disk coordinate, this proves the open-book assertion.

The binding is contact by Theorem~\ref{weightedcontact}, with its
one-dimensional interpretation when $n=2$, or directly by the
Liouville argument in Theorem~\ref{polarcontact}. On each page
interior the Reeb field is positively transverse, since
$d(\arg f_t)(Y)=1$, so $d\beta$ is symplectic there.
For the orientation at the binding, Lemma~\ref{blocklemma} in the
one-row case gives a positive normal Jacobian: on the symplectic
normal plane of $f_t^{-1}(0)$, the map $df_t$ preserves orientation.
Consequently the normal disk orientation is the complex orientation
of $f_t$, and the positive angular direction makes the contact
orientation of $K_{t,r}$ agree with the boundary orientation of the
pages. This verifies all the adaptation conditions.

Finally, the forms
\[
 \beta_s=2\sum_j\bigl((1-s)a_j+s\bigr)
                  (x_j\,dy_j-y_j\,dx_j)|_{S_r},\qquad 0\le s\le1,
\]
are contact: their differentials are symplectic and have the same
transverse radial Liouville field. Gray stability identifies
$\ker\beta$ with the standard contact structure $\ker\beta_1$.
\end{proof}

\begin{Rem}
Oka~\cite[Theorem 22]{Oka2014} proved compatibility with the canonical
contact structure for mixed functions obtained from convenient
non-degenerate holomorphic functions by homogeneous mixed coverings.
Theorem~\ref{ambientbook} gives a direct weighted-form statement for
the entire Hamm deformation, with arbitrary coordinate radial
exponents and every radius. It does not assert that the unweighted
standard form itself is adapted when the $a_j$ differ. Nor does
exactness on page interiors assert non-degeneracy of $d\beta$ up to
the binding on the closed pages.
\end{Rem}

\subsection{Contact isotopy throughout the mixed deformation}

\begin{The}\label{comparison}
For each fixed $r>0$, the contact manifolds
$(K_t,\ker\beta_t)$, $0\le t\le1$, are contactomorphic. Under any smooth
trivialization of the link family starting at one endpoint, the pulled-back
contact structures form an isotopy. Their common contactomorphism type
is the canonical contact type of the holomorphic Hamm link $K_1$.
Consequently they are Milnor fillable. For equal polar exponents these
statements apply to the standard spherical contact structures on all
the mixed links, including $t=0$.
\end{The}
\begin{proof}
Theorem~\ref{isofiber} identifies the compact links smoothly over the
closed interval. The forms $\beta_t$ pull back to a smooth family of contact
forms on one fixed manifold by Theorem~\ref{weightedcontact}. Gray's
stability theorem gives a contact isotopy. Since the regularity of the link
family and the contact condition both hold at $t=0$, the argument covers the
mixed endpoint.

It remains to identify the weighted contact structure at $t=1$. Put
$V=F^{-1}(0)\setminus\{0\}$, a complex manifold, and
$\rho_{\mathbf a}=\sum_j a_j|z_j|^2$.
With $d^c u=du\circ J$, the restricted forms satisfy
\[
 \alpha_{\mathbf a}|_V=-d^c(\rho_{\mathbf a}|_V),\qquad
 \omega_{\mathbf a}|_V=-dd^c(\rho_{\mathbf a}|_V).
\]
The latter is a K\"ahler form. Let $Z$ be the Liouville field determined by
$\iota_Z\omega_{\mathbf a}|_V=\alpha_{\mathbf a}|_V$.
If $g(u,v)=\omega_{\mathbf a}(u,Jv)$, then
$Z=\operatorname{grad}_g(\rho_{\mathbf a}|_V)$. Moreover the calculation
following Theorem~\ref{weightedcontact} gives $d\rho(Z)=\rho>0$.
In particular $Z$ is nonzero, and
\[
 d\rho_{\mathbf a}(Z)=g(Z,Z)>0.
\]
For $0\le s\le1$, put $\rho_s=(1-s)\rho+s\rho_{\mathbf a}$.
Every hypersurface $V\cap\{\rho_s=r^2\}$ is transverse to $Z$, since
\[
 d\rho_s(Z)=(1-s)\rho+s\,g(Z,Z)>0.
\]
These hypersurfaces form a smooth compact family, contained in a fixed
annulus, and the restrictions of $\alpha_{\mathbf a}$ are contact.
Gray stability identifies its two endpoints. At $s=1$ the link and form
are precisely the strictly plurisubharmonic link defined by
$\rho_{\mathbf a}$, so its contact type is the canonical holomorphic one
by \cite{Varchenko1980,Caubel2006}. At $s=0$ they are $(K_1,\beta_1)$.
This proves the identification. For arbitrary large $r$, one may first
decrease $r$ within the contact link family to a sufficiently small radius.
The final assertion follows from Corollary~\ref{equalexponents}.
\end{proof}

Milnor fillability here means contactomorphism to the canonical contact
boundary of an isolated complex singularity. For the holomorphic Hamm
ICIS a smoothing exists, so a Milnor fiber also gives a Stein filling.
The definition of Milnor fillability for general complex singularities
does not require smoothability.

\subsection{Exact symplectic rigidity}\label{s4rigidity}

The polar symmetry strengthens the preceding contact comparison to an
equivalence preserving the contact forms themselves. Throughout this
subsection put
\[
 \Gamma_{\mathbf a}=\prod_{j=1}^n\mu_{a_j},\qquad
 Y=\sum_j a_j^{-1}(x_j\partial_{y_j}-y_j\partial_{x_j}),
 \qquad \rho=\sum_j|z_j|^2,
\]
where $\mu_a$ is the group of $a$-th roots of unity, acting on its
coordinate. The flow of $Y$ defines a circle action with period
$2\pi\operatorname{lcm}(a_1,\ldots,a_n)$.

\begin{Lem}\label{strictmoser}
Let $\beta_t$, $0\le t\le1$, be a smooth family of contact forms on a
closed manifold, all with the same Reeb field $R$. There is an isotopy
$\psi_t$, starting at the identity, such that $\psi_t^*\beta_t=\beta_0$.
If a compact group preserves every $\beta_t$, the isotopy is equivariant.
\end{Lem}
\begin{proof}
Differentiating $\beta_t(R)=1$ gives $\dot\beta_t(R)=0$. Nondegeneracy on
$\ker\beta_t$ therefore gives a unique smooth field $X_t\in\ker\beta_t$
satisfying $\iota_{X_t}d\beta_t=-\dot\beta_t$ on the whole tangent bundle.
Its flow exists on the compact parameter interval, and Cartan's formula gives
\[
 \frac d{dt}\psi_t^*\beta_t
 =\psi_t^*(\dot\beta_t+\iota_{X_t}d\beta_t+d(\beta_t(X_t)))=0.
\]
Uniqueness makes $X_t$ invariant under every symmetry of the family.
\end{proof}

\begin{The}[Exact symplectic rigidity]\label{exactrigidity}
For the mixed Hamm family, assume $n\ge k+2$ and put
$V_t^*=G_t^{-1}(0)\setminus\{0\}$ and
$\lambda_t=\alpha_{\mathbf a}|_{V_t^*}$. There is a smooth family of
diffeomorphisms $\Phi_t:V_0^*\to V_t^*$, with $\Phi_0=\operatorname{id}$, such that
\[
 \Phi_t^*\lambda_t=\lambda_0,\qquad \rho\circ\Phi_t=\rho.
\]
These maps are equivariant for the polar circle action and
$\Gamma_{\mathbf a}$, preserve all coordinate strata, and extend by
$\Phi_t(0)=0$ to homeomorphisms of the zero fibers. Thus they are exact
symplectomorphisms for the restrictions of $\omega_{\mathbf a}$, and,
simultaneously for every $r>0$, induce strict contactomorphisms
$(K_{0,r},\beta_{0,r})\to(K_{t,r},\beta_{t,r})$.
\end{The}
\begin{proof}
Fix $r_*>0$. The compact total link family of Theorem~\ref{isofiber} is
invariant under both actions: $G_t$ is invariant under $\Gamma_{\mathbf a}$
and has polar degree one. Average a lift of $\partial_t$ over the compact
group generated by these actions. The average is still tangent to the
total family and projects to $\partial_t$. Its flow gives equivariant
trivializations $f_t:K_{0,r_*}\to K_{t,r_*}$ with $f_0=\operatorname{id}$.

The forms $\widehat\beta_t=f_t^*\beta_{t,r_*}$ have the common Reeb field
$Y|_{K_{0,r_*}}/(2r_*^2)$, by equivariance and
Theorem~\ref{weightedcontact}. Lemma~\ref{strictmoser} yields an equivariant
isotopy $\psi_t$ for which
$u_t=f_t\circ\psi_t$ satisfies $u_t^*\beta_{t,r_*}=\beta_{0,r_*}$.
Since $a_j\ge2$, the fixed set of the $j$-th factor $\mu_{a_j}$ is
$\{z_j=0\}$. Equivariance of $u_t$ and its inverse consequently preserves
every coordinate subspace and every exact-support stratum.

Let $Z_t$ be the Liouville field of $\lambda_t$. The identity established
after Theorem~\ref{weightedcontact} and Cartan's formula give
\[
 d\rho(Z_t)=\rho,\qquad \mathcal L_{Z_t}\lambda_t=\lambda_t.
\]
Its flow $\varphi_t^s$ exists for every $s\in\re$: along a finite flow
interval $\rho$ stays in a compact interval in $(0,\infty)$, and the
trajectory lies in a compact annulus of the smooth closed fiber $V_t^*$.
The same reasoning is uniform for $t$ in the compact parameter interval.
Hence
\[
 \rho\circ\varphi_t^s=e^s\rho,\qquad
 (\varphi_t^s)^*\lambda_t=e^s\lambda_t.
\]
The map $C_t:\re\times K_{t,r_*}\to V_t^*$,
$C_t(s,x)=\varphi_t^s(x)$, is a diffeomorphism. Since
$\lambda_t(Z_t)=0$, it satisfies $C_t^*\lambda_t=e^s\beta_{t,r_*}$.
Define
\[
 \Phi_t=C_t\circ(\operatorname{id}_{\re}\times u_t)\circ C_0^{-1}.
\]
This proves both required identities and smooth dependence on $t$.
Uniqueness of the Liouville field makes its flow equivariant, so $\Phi_t$
has the stated symmetry and stratum properties. Finally, preservation of
$\rho$ makes $\Phi_t$ and its inverse continuous at the origin.
\end{proof}

\begin{Cor}[Independence of radial exponents and coefficients]\label{coefficientrigidity}
Fix $n\ge k+2$ and $\mathbf a$ with $a_j\ge2$. For any nonnegative integer
vectors $\mathbf b,\mathbf b'$ and any matrices $\Lambda,\Lambda'$ whose
maximal minors are nonzero, the punctured zero fibers of
\[
 \Lambda(z_j^{a_j}|z_j|^{2b_j})_j
 \quad\text{and}\quad
 \Lambda'(z_j^{a_j}|z_j|^{2b'_j})_j
\]
are related by a $\Gamma_{\mathbf a}$- and polar-equivariant
diffeomorphism preserving $\rho$ and the restriction of $\alpha_{\mathbf a}$.
In particular, their strict contact types at each fixed radius depend only
on $n,k,\mathbf a$, and all coordinate strata are preserved.
\end{Cor}
\begin{proof}
Theorem~\ref{exactrigidity} reduces both fibers to their holomorphic
endpoints. The space $\mathcal U_{k,n}$ of matrices with all maximal
minors nonzero is smoothly path connected. Indeed, on the complex affine
line $\Lambda+u(\Lambda'-\Lambda)$ every maximal minor is a polynomial
nonzero at $u=0,1$; a smooth path from $0$ to $1$ avoiding their finitely
many zeros gives a path in $\mathcal U_{k,n}$.

Along such a matrix path $\Lambda_t$, the maps
$F_t=\Lambda_t(z_j^{a_j})_j$ have regular punctured zero fibers transverse
to every sphere, by the proof of Lemma~\ref{hammregular}. Their weighted
contact forms have the same polar Reeb field, and the total link family
is a compact smooth submersion over the interval. The equivariant
trivialization, strict Moser argument and Liouville extension in the proof
of Theorem~\ref{exactrigidity} apply unchanged. Composition with the two
Hamm deformations proves the result.
\end{proof}

\begin{Cor}[Polynomial radial factors]\label{radialprofiles}
Let $A_j(s)$ be nonzero real polynomials with nonnegative coefficients,
and let $\Lambda$ have all maximal minors nonzero. The map
\[
 H(z)=\Lambda\bigl(z_j^{a_j}A_j(|z_j|^2)\bigr)_j,
 \qquad a_j\ge2,\quad n\ge k+2,
\]
defines a mixed ICIS. Its punctured zero fiber is equivariantly
diffeomorphic to that of $\Lambda(z_j^{a_j})_j$, preserving $\rho$,
$\alpha_{\mathbf a}$ and every coordinate stratum. Its spherical links
therefore have the same strict contact type as the corresponding
holomorphic Hamm links.
\end{Cor}
\begin{proof}
Set $A_{j,t}(s)=(1-t)A_j(s)+t$. For $s>0$,
\[
 A_{j,t}(s)>0,\qquad
 a_jA_{j,t}(s)+2sA'_{j,t}(s)>0,
\]
and the determinant of the real differential of
$z_j^{a_j}A_{j,t}(|z_j|^2)$ is
\[
 a_js^{a_j-1}A_{j,t}(s)
       \bigl(a_jA_{j,t}(s)+2sA'_{j,t}(s)\bigr)>0.
\]
These are precisely the positivity conditions used in
Lemmas~\ref{hammregular} and~\ref{symplecticlemma}, so all punctured zero
fibers are regular and symplectic, and their weighted links are contact
with the fixed polar Reeb field. For each coordinate, the radial modulus
$r^{a_j}A_{j,t}(r^2)$ increases from zero to infinity. Lifting a nonzero
vector in $\ker\Lambda$ through these coordinate maps shows that the zero
fibers are nonempty away from the origin; regularity gives their positive
dimension. The origin is critical because $a_j\ge2$.
The smooth total link family and the Liouville identity are therefore
available, and the proof of Theorem~\ref{exactrigidity} applies.
\end{proof}

\begin{Rem}
These equivalences preserve the weighted form $\alpha_{\mathbf a}$,
conjugate the Reeb flows with their time parametrizations, and descend to
symplectomorphisms of the polar quotient orbifolds. For equal polar
exponents they also preserve the standard form $\alpha=\alpha_{\mathbf a}/a$.
They are smooth away from the origin and preserve the coordinate bindings,
but need not identify the angular fibrations of the coordinate open books.
\end{Rem}

\subsection{Ambient contact rigidity}

\begin{Cor}[Ambient contact isotopy]\label{ambientrigidity}
Under the hypotheses of Theorem~\ref{exactrigidity}, fix $r>0$ and write
$\beta=\alpha_{\mathbf a}|_{S_r^{2n-1}}$. Let
$u_t=\Phi_t|_{K_{0,r}}:K_{0,r}\to K_{t,r}$. There is a contact isotopy
\[
 \Psi_t:(S_r^{2n-1},\ker\beta)\longrightarrow
        (S_r^{2n-1},\ker\beta),\qquad \Psi_0=\operatorname{id},
\]
such that $\Psi_t|_{K_{0,r}}=u_t$. In particular,
$\Psi_t(K_{0,r})=K_{t,r}$.
For fixed $n,k,\mathbf a$ and $r$, the ambient contact isotopy class
of the mixed Hamm link is independent of the radial exponents and
of the coefficient matrix with nonzero maximal minors.
\end{Cor}

\begin{proof}
The sphere is contact for $\beta$, since the radial Liouville field
of $d\alpha_{\mathbf a}$ is transverse to it. If $\iota_t$ denotes
inclusion, the maps $j_t=\iota_t\circ u_t$ satisfy
$j_t^*\beta=\beta|_{K_{0,r}}$, and hence form an isotopy of contact
embeddings. Contact isotopy extension gives the ambient assertion;
compare \cite[Theorem 2.42]{GeigesContact}.

For completeness, the strictness here allows the prescribed
parametrization to be retained by a direct Hamiltonian extension.
Let $v_t=\dot j_t\circ j_t^{-1}$ be the velocity along $K_{t,r}$ and
put $h_t=\beta(v_t)$ there. Differentiating $j_t^*\beta$ gives
\[
 dh_t(w)=-d\beta(v_t,w),\qquad w\in TK_{t,r}.
\]
Thus the value $h_t$ and the ambient covector
$-\,\iota_{v_t}d\beta$ constitute a compatible first jet along the
embedded link. Extend this jet smoothly to a time-dependent function
$H_t$ on the sphere, using tubular neighborhoods and a cutoff.
Its contact Hamiltonian field $X_t$, with conventions
\[
 \beta(X_t)=H_t,\qquad
 \iota_{X_t}d\beta=(R_\beta H_t)\beta-dH_t,
\]
agrees with $v_t$ along the link: there
$dH_t(R_\beta)=-d\beta(v_t,R_\beta)=0$, and the two displayed
equations uniquely determine the vector field. Its flow exists
throughout $[0,1]$ by compactness and satisfies
$\Psi_t\circ j_0=j_t$.

Finally, connect two radial choices through their holomorphic
endpoints and connect their coefficient matrices by the admissible
paths in the proof of Corollary~\ref{coefficientrigidity}.
The strict link isotopies along these paths admit the same extension.
Concatenating the resulting ambient isotopies proves the last assertion.
\end{proof}

\begin{Rem}
The extension preserves the ambient contact structure:
$\Psi_t^*\beta=q_t\beta$ for a positive function $q_t$, with
$q_t|_{K_{0,r}}=1$. We do not assert strict preservation of $\beta$
on the whole sphere, preservation of all ambient coordinate strata,
or identification of the angular fibrations. Each radius is fixed
separately in this corollary.
\end{Rem}

\subsection{Giroux equivalence in dimension three}

Recall that a positive stabilization of a surface open book attaches
a one-handle to its page and composes its monodromy with a positive
Dehn twist along a curve passing once through that handle.
Giroux's correspondence~\cite{Giroux2002} states that open books
supporting isotopic positive cooriented contact structures on a fixed
closed oriented three-manifold admit isotopic positive stabilizations.
The following consequences use precisely this three-dimensional theorem.

\begin{Cor}[Common positive stabilizations]\label{girouxstable}
\hfill
\begin{enumerate}
\item In Theorem~\ref{ambientbook}, let $n=2$ and fix $r$ and
$\mathbf a$. All the ambient open books, for all $t$, radial exponents
$\mathbf b$ and nonzero coefficients $\lambda_j$, admit pairwise
common positive stabilizations on $S_r^3$. In particular, each is
positively stably equivalent to the holomorphic open book of $f_1$.
\item Let $k=1$, $n=3$, and fix a mixed Hamm deformation and a radius.
Identify each contact link $(K_{t,r},\ker\beta_t)$ with the holomorphic
contact link $(K_{1,r},\ker\beta_1)$ using
Theorem~\ref{exactrigidity}. After transport by these identifications,
any two of the coordinate open books of Theorem~\ref{coordbook}, or
the additional-function open books of Corollary~\ref{hammmilnorbook},
admit common positive stabilizations. This includes comparison with
their holomorphic counterparts at $t=1$.
\end{enumerate}
\end{Cor}

\begin{proof}
In the first case every open book supports the same contact structure
$\ker(\alpha_{\mathbf a}|_{S_r^3})$ by
Theorem~\ref{ambientbook}. In the second case the transported open books
all support $\ker\beta_1$, by the adaptation theorems and the contact
identifications. Apply Giroux's correspondence on the fixed
three-manifold in each case.
\end{proof}

\begin{Rem}
These conclusions concern equivalence after positive stabilization.
They do not assert that the strict contactomorphisms of
Theorem~\ref{exactrigidity} identify the original angular fibrations,
nor that the original page symplectic structures or symplectic
monodromies are equivalent. The ambient case $n=2$ and the link case
$n=3$ are distinct applications; no higher-dimensional stabilization
assertion is needed here.
\end{Rem}

\subsection{Examples}

\begin{Exam}
For $k=2$, $n=4$, take the columns of $\Lambda$ to be $(1,c_j)^T$,
where the four numbers $c_j$ are distinct, and put $a_j=2$ for all $j$.
For arbitrary nonnegative integers $b_j$, every link of
\[
 \sum_{j=1}^4z_j^2\bigl(t+(1-t)|z_j|^{2b_j}\bigr)=0,\qquad
 \sum_{j=1}^4c_jz_j^2\bigl(t+(1-t)|z_j|^{2b_j}\bigr)=0
\]
is a three-dimensional contact submanifold of the standard sphere, of the
contact type of the holomorphic complete intersection of the two quadrics,
including the mixed endpoints with different radial exponents. Each
coordinate argument gives an adapted open book.
\end{Exam}

\begin{Exam}
With the same matrix and arbitrary integers $a_j\ge2$, the same holds for
the weighted form $\alpha_{\mathbf a}$, with holomorphic comparison link
$\sum z_j^{a_j}=\sum c_jz_j^{a_j}=0$. If $b_j\ge2$ for all $j$, the mixed
endpoint is strictly contactomorphic to this holomorphic link, yet by
Proposition~\ref{algicis} it is not finitely determined in the real smooth
contact category.
\end{Exam}

\subsection{The standard form on ellipsoidal links}

For unequal polar exponents the standard form $\alpha$ is still a contact
form, provided the spheres are replaced by the ellipsoids adapted to the
polar weights. Put
\[
 \rho_{1/\mathbf a}=\sum_j\frac{|z_j|^2}{a_j},\qquad
 E_r=\{\rho_{1/\mathbf a}=r^2\},\qquad
 Y_{\mathbf a}=\sum_j\frac1{a_j}(x_j\partial_{y_j}-y_j\partial_{x_j}).
\]

\begin{Prop}\label{ellipsoidal}
Let $G_t$ be the mixed Hamm family, with $n\ge k+2$. For every $t\in[0,1]$
and $r>0$, the zero fiber $V_t$ is transverse to $E_r$, and the standard form
$\alpha$ restricts to a contact form on $L_{t,r}=V_t\cap E_r$, with Reeb field
$Y_{\mathbf a}/(2r^2)$. The contact manifolds $(L_{t,r},\ker\alpha)$ are
contactomorphic to the canonical contact structure of the holomorphic Hamm
link $K_1$.
\end{Prop}

\begin{proof}
By Lemma~\ref{hammregular} and Lemma~\ref{symplecticlemma} with $c_j=1$,
$V_t\setminus\{0\}$ is symplectic for $\omega=d\alpha$. The field
$Y_{\mathbf a}$ is tangent to $V_t$ and to the ellipsoids, and
\[
 \iota_{Y_{\mathbf a}}\omega=-2\,d\rho_{1/\mathbf a},\qquad
 \alpha(Y_{\mathbf a})=2\rho_{1/\mathbf a}.
\]
Let $Z_t$ be the Liouville field of $\alpha|_{V_t}$. Evaluating
$\iota_{Z_t}\omega=\alpha$ on $Y_{\mathbf a}$ gives
$d\rho_{1/\mathbf a}(Z_t)=\rho_{1/\mathbf a}>0$, which proves transversality.
The contact assertion and the Reeb field follow as in the proof of
Theorem~\ref{weightedcontact}. The proof of Theorem~\ref{isofiber} applies with
$\rho$ replaced by $\rho_{1/\mathbf a}$, since the field $W$ used there also
satisfies $d\rho_{1/\mathbf a}(W)>0$; with Gray stability, this identifies all
the contact manifolds $(L_{t,r},\ker\alpha)$ with $(L_{1,r},\ker\alpha)$.

At $t=1$, $\alpha|_{V_1}=-d^c(\rho|_{V_1})$, and as in the proof of
Theorem~\ref{comparison} the Liouville field $Z_1$ is the gradient of
$\rho|_{V_1}$ for the K\"ahler metric, so $d\rho(Z_1)>0$. Hence
$d\rho_s(Z_1)>0$ for every convex combination
$\rho_s=(1-s)\rho_{1/\mathbf a}+s\rho$, and Gray stability identifies
$(L_{1,r},\ker\alpha)$ with the spherical link of $V_1$ with the form
$\alpha=-d^c\rho$, whose contact structure is the canonical one
\cite{Varchenko1980}.
\end{proof}

The spherical link $K_t$ is contact for $\alpha$ whenever
$d\rho(Z_t)>0$ along $K_t$. To compare it with an ellipsoidal link by convex
interpolation, positivity must hold throughout the region swept out by the
intermediate level sets. In particular, if $d\rho(Z_t)>0$ on a punctured
neighborhood of the origin in $V_t$, then, for sufficiently small $r$,
all levels $V_t\cap\{(1-s)\rho+s\rho_{1/\mathbf a}=r^2\}$ lie in that
neighborhood and are transverse to $Z_t$, since
\[
 d\bigl((1-s)\rho+s\rho_{1/\mathbf a}\bigr)(Z_t)
 =(1-s)d\rho(Z_t)+s\rho_{1/\mathbf a}>0.
\]
They form a compact smooth family of contact hypersurfaces, so Gray
stability identifies their contact structures. This positivity holds at
$t=1$, and for every $t$ when the exponents $a_j$ are equal, since then
$\rho_{1/\mathbf a}=\rho/a$. The remaining question is therefore the following.

\begin{Qst}
Let $Z_t$ be the Liouville field of $\alpha|_{V_t}$ for a mixed Hamm family
with unequal polar exponents. Is $d\rho(Z_t)>0$ on $V_t\setminus\{0\}$, at
least near the origin?
\end{Qst}

Since $\alpha(v)=d\alpha(Z_t,v)$ on $TV_t$, the restriction of $\alpha$ to
$K_t$ is contact exactly when $Z_t$ is nowhere tangent to $K_t$, that is,
when $d\rho(Z_t)\ne0$ along $K_t$. The question therefore asks whether the
standard form is a positive contact form on the small round links.

\section*{Acknowledgments}
We thank Patrick Popescu-Pampu for useful discussions about contact
structures and the geometric issues considered here. We are also grateful
to Maria Aparecida Ruas (\textit{in memoriam}), Cidinha, for valuable
conversations and for her inspiring example as both a human being and
a mathematician.

\subsection*{Conflict of interest}
On behalf of all authors, the corresponding author states that there is
no conflict of interest.
\subsection*{Data availability}
Data availability is not applicable to this article.

\end{document}